\documentclass[leqno,final]{siamltex}
\usepackage{amsmath,amsfonts}
\usepackage{graphicx}
\usepackage{float}
\usepackage{subcaption}
\usepackage{appendix}
\usepackage{cleveref}
\usepackage{epstopdf}

\newcommand{\R}{\mathbb{R}}
\newcommand{\C}{\mathbb{C}}
\newcommand{\II}{\mathcal{I}}
\newcommand{\Ha}{\mathcal{H}}

\usepackage{color}

\newcommand{\norm}[1]{\left\lVert#1\right\rVert}

\newtheorem{example}{Example}[section]

\usepackage{tikz}
\usepackage[siunitx,RPvoltages]{circuitikz}

\title{On non-Hermitian positive (semi)definite linear algebraic systems arising from dissipative Hamiltonian DAE\MakeLowercase{s} \thanks{Version of \today.}}

\author{Candan G{\"u}d{\"u}c{\"u}\footnotemark[2] \and J{\"o}rg Liesen\footnotemark[2]
\and Volker Mehrmann\footnotemark[2] \and Daniel B. Szyld\footnotemark[4]}

\begin{document}
	
\maketitle

\renewcommand{\thefootnote}{\fnsymbol{footnote}}

\footnotetext[2]{Institute of Mathematics, TU Berlin, Stra{\ss}e des 17.~Juni 136, 10623 Berlin, Germany.  Email: {\tt \{guducu,liesen,mehrmann\}@math.tu-berlin.de}.}
\footnotetext[4]{Department of Mathematics, Temple University, 1805 N.~Broad Street, Philadelphia, PA 19122, USA. Email: {\tt szyld@temple.edu}.}

\begin{abstract}
We discuss different cases of dissipative Hamiltonian differential-algebraic equations and the linear algebraic systems that arise in their linearization or discretization. For each case we give examples from practical applications. An important feature of the linear algebraic systems is that the (non-Hermitian) system matrix has a positive definite or semidefinite Hermitian part.
In the positive definite case we can solve the linear algebraic systems iteratively by Krylov subspace methods based on efficient three-term recurrences. We illustrate the performance of these iterative methods on several examples. The semidefinite case can be challenging and requires additional techniques to deal with ``singular part'', while the ``positive definite part'' can still be treated with the three-term recurrence methods.
\end{abstract}

\begin{keywords}
dissipative Hamiltonian system, port-Hamiltonian system, descriptor system, differential-algebraic equation, linear algebraic system, positive semidefinite Hermitian part, Krylov subspace method
\end{keywords}

\begin{AMS}
65L80, 65F10, 93A15, 93B11, 93B15
\end{AMS}

\section{Introduction}

It is well known that every matrix $A\in\C^{n,n}$ can be split into its Hermitian and skew-Hermitian parts, i.e.,
\begin{equation}\label{eqn:split}
A=H+S,\quad H=\frac12 (A+A^*)\quad\mbox{and}\quad S=\frac12 (A-A^*),
\end{equation}
where $A^*$ is the Hermitian transpose (or the transpose in the real case) of $A$, so that $H=H^*$ and $S=-S^*$. This simple, yet fundamental observation has many useful applications. For example, Householder used it in~\cite[p.~69]{Hou64} to show that all eigenvalues of $A=H+S$ lie in or on the smallest rectangle with sides parallel to the real and imaginary axes that contains all eigenvalues of $H$ and of $S$. This result is attributed to Bendixson~\cite{Ben1902}, and was refined by Wielandt~\cite{Wie55}. It shows that if $H$ is positive definite, then all eigenvalues of $A$ have a positive real part, and therefore such (in general non-Hermitian) matrices $A$ are sometimes called \emph{positive real}. Here we call $A=H+S$ \emph{positive definite} or \emph{positive semidefinite} if $H$ has the corresponding property.

Our first goal in this paper is to show that, while every matrix $A\in\C^{n,n}$ trivially splits into $A=H+S$, there is an important class of practically relevant applications where this splitting \emph{occurs naturally} and has a \emph{physical meaning}. The class of applications we consider is given by energy-based modeling using differential algebraic equation (DAE) systems in dissipative Hamiltonian (dH) form, or for short dHDAE systems. The applicability of this modeling approach has been demonstrated in a variety of
application areas
such as thermodynamics, electromagnetics, fluid mechanics, chemical processes, and general optimization; see, e.g.,~\cite{EbeM04,GayY19,HamLM07,HamLM08,HauMMMRS19,StePS15}. Properties of dHDAE systems have been studied in numerous recent publications;
see, e.g.,~\cite{BeaMV19,BeaMXZ18,GilMS18,MehMS16,MehMS18,MehMW18,MehMW21,MehM19,Sch13}.

We systematically discuss different cases of linear and constant-coefficient dHDAE systems, and we illustrate these cases with examples from practical applications. The linear algebraic systems that arise from the linearization and/or discretization of the dHDAE systems are of the form $A=H+S$, where the Hermitian part $H$ (and hence~$A$) is positive definite or at least positive semidefinite.

We also discuss how to solve the linear algebraic systems arising from dHDAE systems. In the positive definite case, Krylov subspace methods based on efficient three-term recurrences can be used. The semidefinite case can be challenging and typically requires additional techniques that deal with the ``singular part'' of $H$, while the ``positive definite part'' of $H$ still allows an application of three-term recurrence methods. We show that the formulation of the dHDAE system often leads to a linear algebraic system where the ``singular part'' of $H$ can be identified without much additional effort. For problems where this is not the case we show how on the linear algebraic level the ``singular part'' of $H$ can be isolated and dealt with using a unitary congruence transformation to a \emph{staircase form}, and further via Schur complement reduction to a block diagonal form.

The paper is organized as follows.
In Section~\ref{sec:dHDAE} we introduce the standard form of linear and constant-coefficient dHDAE systems, and in Section~\ref{sec:examples} we give a systematic overview of the different cases of these systems. In Section~\ref{sec:lin-sys} we discuss the form of linear algebraic systems arising from the time-discretization of dHDAE systems, and we describe a staircase form for these systems. In Section~\ref{sec:solvers} we discuss iterative methods based on three-term recurrences for the discretized systems, and in Section~\ref{sec:tests} we present numerical examples with these methods applied to different cases of dHDAE systems. The paper ends with concluding remarks in Section~\ref{sec:conclusion}.


\section{Linear dissipative Hamiltonian DAE systems}\label{sec:dHDAE}

The standard form of a linear dHDAE system, where for simplicity we consider the case of constant (i.e., time-invariant) coefficients, is given by
\begin{eqnarray}
E\dot x &=& (J-R)Qx +f,\label{eqn:dHDAE-1}\\
x(t_0) &=& x^0; \label{eqn:dHDAE-iv}
\end{eqnarray}
see~\cite{BeaMXZ18,MehMW18}, where this class is introduced and studied in the context of control problems for port-Hamiltonian (pH) systems.
The physical properties of the modeled system are encoded in the algebraic structure of the coefficient matrices.
The matrix $E\in\C^{n,n}$ is called \emph{flow matrix}, the skew-Hermitian \emph{structure matrix} $J\in\C^{n,n}$ describes the energy flux among energy storage elements, the Hermitian positive semidefinite \emph{dissipation matrix} $R\in\C^{n,n}$ describes energy loss and/or dissipation.
The energy function or \emph{Hamiltonian} associated with the system \eqref{eqn:dHDAE-1} is given by the function
\[
\Ha(x)=\frac12 (x^*Q^*Ex),
\]
and typically, since this is an energy, one has that
\begin{equation}\label{eqn:dHDAE-3}
E^*Q=Q^*E\geq 0, 
\end{equation}
where $H\geq 0$ means that the Hermitian matrix $H$ is positive semidefinite. Note that~\eqref{eqn:dHDAE-3} implies that $\Ha(x)\geq 0$ for all states $x$.

Linear dHDAE systems of the form \eqref{eqn:dHDAE-1} often arise directly in mathematical modeling, or as a result of linearization along a stationary solution for general, nonlinear dHDAE systems; see, e.g.,~\cite{MehM19}. In many applications, furthermore, the matrix~$Q$ is the identity, and if not, it can be turned into an identity for a subsystem; see~\cite[Section~6.3]{MehMW21}. Thus, in the following we restrict ourselves to dHDAE systems of the form
\begin{equation}\label{eqn:dHDAE-final}
E\dot x=(J-R)x+f,\quad\mbox{where}\quad E=E^*\geq 0,\quad J=-J^*,\quad R=R^*\geq 0.
\end{equation}
For analyzing the system \eqref{eqn:dHDAE-final} it is useful to transform it into a \emph{staircase form} that reveals its ``positive definite part'' and its ``singular part'', as well as the common nullspaces (if any) of the different matrices. Such a form was derived using a sequence of spectral and singular value decompositions in~\cite[Lemma~5]{AchAM21}, and is adapted here to our notation.
%
\begin{lemma} \label{lem:regul}
For every dHDAE system of the form \eqref{eqn:dHDAE-final} there exists a unitary (basis transformation)
matrix $\tilde V\in \mathbb C^{n, n}$, such that the system in the new variable
$$
\tilde{x}=\tilde{V}^* x=\begin{bmatrix} \tilde{x}_1 \\ \tilde{x}_2 \\ \tilde{x}_3  \\ \tilde{x}_4\\ \tilde{x}_5\end{bmatrix} 		
	\begin{tabular}{l}
	$\left.\lefteqn{\phantom{\begin{matrix} \end{matrix}}}\hspace*{-4mm} \right\}n_1$\\
	$\left.\lefteqn{\phantom{\begin{matrix}  \end{matrix}}}\hspace*{-4mm} \right\}n_2$\\
	$\left.\lefteqn{\phantom{\begin{matrix} \end{matrix}}}\hspace*{-4mm} \right\}n_3$\\
	$\left.\lefteqn{\phantom{\begin{matrix} \end{matrix}}}\hspace*{-4mm} \right\}n_4$\\
	$\left.\lefteqn{\phantom{\begin{matrix} \end{matrix}}}\hspace*{-4mm} \right\}n_5$
\end{tabular}
$$
has the $5\times 5$ block form
\begin{align}
&\begin{bmatrix}
E_{11} & E_{12} & 0 & 0 &0\\
E_{21} & E_{22} & 0 & 0 &0\\
0 & 0 & 0 &0 &0\\
0 & 0 & 0 &0 &0\\
0 & 0 & 0 &0&0\end{bmatrix}
\begin{bmatrix} \dot{\tilde{x}}_1 \\ \dot{\tilde{x}}_2 \\ \dot{\tilde{x}}_3  \\ \dot{\tilde{x}}_4\\ \dot{\tilde{x}}_5\end{bmatrix} =\label{stc1} \\
& \begin{bmatrix}
J_{11}-R_{11} &  J_{12}-R_{12}& J_{13}-R_{13} & J_{14}&0\\
J_{21}-R_{21} & J_{22}-R_{22} & J_{23}-R_{23} & 0 &0\\
J_{31}-R_{31} & J_{32}-R_{32} & J_{33} -R_{33} & 0 & 0\\
J_{41} & 0 & 0  & 0 & 0\\
0& 0 &  0 & 0& 0\end{bmatrix} \begin{bmatrix} \tilde{x}_1 \\ \tilde{x}_2 \\ \tilde{x}_3  \\ \tilde{x}_4\\ \tilde{x}_5\end{bmatrix}+\begin{bmatrix} f_1(t) \\ f_2(t) \\ f_3(t)  \\ f_4(t) \\ f_5(t) \end{bmatrix},\label{stc2}
\end{align}
where $n_1,n_2,n_3,n_4,n_5\in {\mathbb N}_0$, and $n_1=n_4$. If it is present in \eqref{stc1}, the matrix $\begin{bmatrix} E_{11}& E_{12}\\ E_{21} & E_{22} \end{bmatrix}$ (or just $E_{22}$ if $n_1=n_4=0$) is Hermitian positive
definite, and if they are present in \eqref{stc2}, the matrices $J_{33}-R_{33}$ and $J_{41}=-J_{14}^*$ are nonsingular.
\end{lemma}

>From the staircase form \eqref{stc1}--\eqref{stc2} we immediately see that the initial value problem \eqref{eqn:dHDAE-final} is uniquely solvable (for consistent initial values and sufficiently often differentiable inhomogeneities $f$) if and only if the last block row and column in the matrices (which contain only zeros) do not occur, i.e., if $n_5=0$. If $n_5\neq 0$, then $\tilde{x}_5$ can be chosen arbitrarily. In the following we assume that $n_5=0$, i.e., we assume throughout that~\eqref{eqn:dHDAE-final} is uniquely solvable. Equivalently, we assume that the pencil $\lambda E-(J-R)$ is regular.

As shown in~\cite[Corollary~1]{AchAM21}, the differentiation index (i.e., the size of the largest Jordan block associated with the eigenvalue $\infty$) of a regular pencil $\lambda E-(J-R)$ in terms of the staircase form \eqref{stc1}--\eqref{stc2} is given by
\begin{align*}
& \mbox{zero if and only if $n_1=n_4=0$ and $n_3=0$ (or simply $n_2=n$),}\\
& \mbox{one if and only if $n_1=n_4=0$ and $n_3>0$,}\\
& \mbox{two if and only if $n_1=n_4>0$.}
\end{align*}
These are all possible cases that can occur. It is easy to see that the positive definite case $E=E^*>0$ in \eqref{eqn:dHDAE-final} corresponds to a staircase form \eqref{stc1}--\eqref{stc2} with $n_2=n$ and hence to the index zero, regardless of the properties of $J$ and $R$. On the other hand, a singular matrix $E=E^*\geq 0$ corresponds to an index either one or two, depending on the relation between the matrices $E,J,R$. Distinguishing between these three cases will be important in our overview in the next section.

In numerical practice, a computation of the form \eqref{stc1}--\eqref{stc2} for a given dHDAE system requires a sequence of nullspace computations, which can be carried out by singular value decompositions.
Unfortunately, these sequences of dependent rank decisions may be very sensitive under perturbations; see, e.g., \cite{ByeMX07} where the construction of similar staircase forms and the challenges are discussed. Also, these factorizations are often not efficiently computable for large-scale problems. However, as we will demonstrate with several examples in the next section, in many cases the structural properties arising from physical modeling help to make this process easier.

\section{Different cases and specific examples}\label{sec:examples}

We will now present a systematic overview of different cases of systems of the forms \eqref{eqn:dHDAE-final} or \eqref{stc1}--\eqref{stc2} that occur in applications, ordered by properties of $E$ and the index of the (regular) pencil $\lambda E-(J-R)$. The examples given in this section demonstrate the large variety of applications for dHDAEs.

\subsection*{Case 1: Positive definite $E$, index zero}
 The case of $E=E^*>0$ in \eqref{eqn:dHDAE-final}, or $n_2=n$ in \eqref{stc1}--\eqref{stc2}, is the ``simplest'' one.
This case usually leads to a positive definite Hermitian part of the coefficient matrix in the linear algebraic system; see Section~\ref{sec:lin-sys} below.
	
\begin{example}[index zero]\label{ex:MDK}{\rm
Consider the classical second order representation of a linear damped mechanical system, which is given by
\begin{equation} \label{MDK_eq}
M\ddot{x} + D\dot{x}+Fx=f,
\end{equation}
where $M,D,F\in\R^{n,n}$ are Hermitian matrices with $M,F>0$ and $D\geq0$; see,
e.g.,~\cite[Chapter 1]{Ves11}.
By introducing the variables, $\hat{x}_2=x$ and $\hat{x}_1 = \dot{x}$, equation \eqref{MDK_eq} can be written as
\begin{equation} \label{eq_brake_dae}
\begin{bmatrix}M &0 \\0 & F \end{bmatrix} \begin{bmatrix} \dot{\hat x}_1 \\ \dot{\hat x}_2 \end{bmatrix} = \Bigg(  \begin{bmatrix}0  & -F \\ F &0  \end{bmatrix} - \begin{bmatrix} D & 0\\ 0&0 \end{bmatrix}\Bigg)
\begin{bmatrix} \hat x_1  \\  \hat x_2 \end{bmatrix} +
	\begin{bmatrix} f \\ 0 \end{bmatrix},
\end{equation}
which is of the form \eqref{eqn:dHDAE-final} with $E=\begin{bmatrix}M &0 \\0 & F \end{bmatrix}=E^*>0$.
}
\end{example}
	
\begin{example}[index zero]\label{ex:poro}{\rm
The discretization of the poroelasticity equations that model the deformation of porous media saturated by an incompressible viscous fluid in first order formulation as in~\cite[Section~3.4]{AltMU21} leads to a dHDAE of the form
\begin{align}
\label{eqn:twoField:opMatrix3}
\begin{bmatrix}Y & 0 & 0\\ 0 & A & 0 \\ 0 & 0 & M\end{bmatrix}
  \begin{bmatrix} \dot{w} \\\dot{u}\\ \dot{p} \end{bmatrix}
  = \left(\begin{bmatrix} 0 & -A & D^* \\ A & 0 & 0 \\ -D & 0 & 0 \end{bmatrix}-
\begin{bmatrix} 0 & 0 &0 \\ 0 & 0 & 0 \\ 0 & 0 & F \end{bmatrix}\right)
  \begin{bmatrix} w\\ u\\ p \end{bmatrix}  + \begin{bmatrix} f\\ 0 \\ g \end{bmatrix},
\end{align}
where $A,M,Y$ are Hermitian positive definite (where $Y$ is of very small norm), $F$ is typically Hermitian positive semidefinite, and $D$ is general, non-Hermitian. Here $u$ represents the discretized displacement field, $w$ the associated discretized velocities, and $p$ the discretized pressure.
Again we have a system of the form \eqref{eqn:dHDAE-final} with $E={\rm diag}(Y,A,M)=E^*>0$.
}
\end{example}

\subsection*{Case 2: Positive semidefinite $E$, index one}
In this case we have a staircase form \eqref{stc1}--\eqref{stc2} with $n_1=n_4=0$ and $n_3\neq 0$, which after renumbering the equations and unknowns can be written as
\begin{equation}\label{case2}
\begin{bmatrix} E_{11} & 0 \\ 0 & 0 \end{bmatrix}
\begin{bmatrix} \dot{x}_1 \\ \dot{x}_2 	\end{bmatrix}=
\begin{bmatrix} J_{11} - R_{11} & J_{12}-R_{12} \\ J_{21}-R_{21} & J_{22}-R_{22}  \end{bmatrix}\begin{bmatrix} 	x_1 \\ x_2 	\end{bmatrix} +f,
\end{equation}
where $E_{11}=E_{11}^*>0$ and where $J_{22}-R_{22}$ is nonsingular. Note that if it is known in advance that the given dHDAE has index one, the form \eqref{case2} can be obtained from \eqref{eqn:dHDAE-final} by a single (unitary) transformation that ``splits off'' the nullspace of $E$.
Whether the coefficient matrix $A=H+S$ of the corresponding linear algebraic system (after time discretization) in this case has a positive definite or semidefinite Hermitian part $H$ depends on the properties of $R_{22}$. The Hermitian part is of the form $H=E+\frac{\tau}{2}R$ (see Section~\ref{sec:lin-sys} below), and hence a positive definite $R_{22}$ will lead to a positive definite $H$, which may be (highly) ill-conditioned, since $R$ is multiplied by the potentially small constant $\tau/2$.

\begin{figure}[h]
		\centering
		\begin{circuitikz}[scale=1.2,/tikz/circuitikz/bipoles/length=1cm]
			\draw (0,0) node[ground] {} to[american controlled voltage source,invert,v>=$E_G$] (0,1) to[R=$R_G$,i>=$\II_G$] (0,2) to (1,2) to[L=$L$,i>=$\II$] (2,2) to[R=$R_L$] (3,2) to (4,2) to[R=$R_R$,i<=$\II_R$] (4,0) node[ground] {};
			\draw (1,2) to[C=$C_1$,v>=$V_1$,i<=$\II_1$,*-] (1,0) node[ground] {};
			\draw (3,2) to[C,l_=$C_2$,v^>=$V_2$,i<_=$\II_2$,*-] (3,0) node[ground] {};
		\end{circuitikz}
		\caption{A simple RLC circuit}\label{fig:circuit}
	\end{figure}
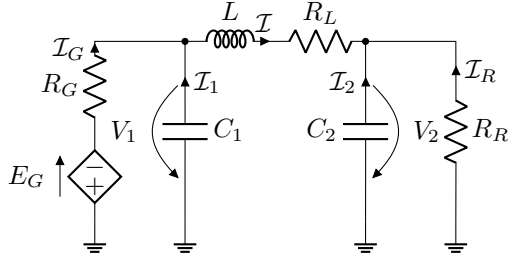
	
\begin{example}[index one]\label{ex:circuit}{\rm
Consider the linear RLC circuit shown in Figure~\ref{fig:circuit} (see~\cite[Example 4.1]{MehM19}), which is modeled by the following equations:
\begin{align*}
			L\dot{\II} &= -R_L \II +V_2 - V_1, \\
			C_1 \dot{V}_1 &= \II - \II_G, \\
			C_2 \dot{V}_2 &= - \II -\II_R, \\
			0 &= -R_G\II_G + V_1 + E_G, \\
			0 &= -R_R\II_R + V_2.
\end{align*}
Here $R_G, R_L, R_R > 0$ represent  resistances, $L > 0$ inductances, $C_1, C_2 > 0$ capacitances, and $E_G$ a controlled voltage source. The equations can be written in the form \eqref{case2} with $E = {\rm diag}(L, C_1 , C_2 , 0, 0)$, the vector of unknowns $ x = [\II^*, V_1^* , V_2^* , I_G^* , I_R^*]^*$,
\[
J= \begin{bmatrix} 	0 & -1 & 1 & 0 & 0 \\ 1 & 0 & 0 & -1 & 0 \\ -1 & 0 & 0 & 0 & -1 \\ 0 & 1 & 0 & 0 & 0 \\ 0 & 0& 1 & 0 & 0
\end{bmatrix},\quad\mbox{and}\quad
R= \begin{bmatrix} 	R_L & 0 & 0 & 0 & 0 \\ 0 & 0 & 0 & 0 & 0 \\ 0 & 0 & 0 & 0 & 0 \\ 0 & 0 & 0 & R_G & 0 \\ 0 & 0& 0 & 0 & R_R
\end{bmatrix},
\]
so that $E_{11}=E_{11}^*={\rm diag}(L,C_1,C_2)>0$, and $J_{22}-R_{22}=-{\rm diag}(R_G,R_R)$ is nonsingular, and the nullspace of $E$ is displayed directly. Note that most RLC circuits (potentially with millions of equations and unknowns) have this index-one structure, but occasionally they have index two \cite{EstT00}.
}
\end{example}

\begin{example}[index one]\label{ex:stokes}
{\rm
The space discretization of the unsteady incompressible Stokes or linearized Navier-Stokes equations via finite element or finite difference methods typically leads to dHDAE systems of the form
\begin{equation}\label{P-instat-op-general}
\begin{bmatrix} M & 0 \\ 0 & 0 \end{bmatrix}
 \begin{bmatrix} \dot{v} \\ \dot{p} \end{bmatrix} =
\left(\begin{bmatrix} A_S & B\\ -B^* & 0\end{bmatrix} -
\begin{bmatrix} -A_H & 0\\ 0 & -C\end{bmatrix}\right)
\begin{bmatrix} v\\p\end{bmatrix}+\begin{bmatrix}f\\g\end{bmatrix},
\end{equation}
where $M=M^*>0$ is the mass matrix, $A_S=-A_S^*$, $B^*$ is the discretized divergence operator (normalized so that it is of full row rank), $-A_H=-A_H^*\geq 0$, and $-C=-C^*>0$ is a stabilization term, typically of small norm; see, e.g., \cite{EmmM13}. In the Stokes case we usually have $A_S=0$. Here $v$ and $p$ denote the discretized velocity and pressure, respectively. In terms of \eqref{case2} we have the Hermitian positive definite matrix $E_{11}=M$, and the nonsingular matrix $J_{22}-R_{22}=-C$.
}
\end{example}

\subsection*{Case 3: Positive semidefinite $E$, index~two}
In this case we have a staircase form \eqref{stc1}--\eqref{stc2} with $n_1=n_4>0$.

\begin{example}[index two] 
\label{ex:poro2}
{\rm
Consider Example~\ref{ex:poro} in the quasi-stationary regime (see \cite{Sho00}), where one usually sets $Y=0$. After a permutation of the block rows, the system has the form
\begin{align}
\label{eqn:twoField:opMatrixstat}
\begin{bmatrix}M & 0 & 0\\ 0 & A & 0 \\ 0 & 0 & 0\end{bmatrix}
  \begin{bmatrix} \dot{p} \\ \dot{u}\\ \dot{w} \end{bmatrix}
  = \left(
    \begin{bmatrix} 0 & 0 & -D \\ 0 & 0 & A \\ D^* & -A & 0 \end{bmatrix}-
  \begin{bmatrix} F & 0 & 0 \\ 0 & 0 & 0 \\ 0 & 0 & 0 \end{bmatrix}\right)
  \begin{bmatrix} p\\ u\\ w \end{bmatrix}  + \begin{bmatrix} g\\ 0 \\ f \end{bmatrix}.
\end{align}
with $A=A^*,M=M^*$ positive definite. The form \eqref{stc1}--\eqref{stc2} with $n_3=0$
is obtained by performing a $QR$ decomposition of the full row rank matrix
$[D^*\, -A]$, and then transforming the system accordingly.
}
\end{example}

\begin{example}[index two]
\label{ex:stokeswostab}
{\rm
Consider Example~\ref{ex:stokes} without stabilization, i.e., with $C=0$. Let $B^* =U_B [\Sigma\; 0] V_B^*$, be a singular value decomposition with unitary matrices $U_B,V_B$, and a nonsingular diagonal matrix $\Sigma$ (corresponding to the splitting of the space of functions into the subspace of divergence free functions and its orthogonal complement). After a unitary similarity transformation we obtain a staircase form \eqref{stc1}--\eqref{stc2} with $n_3=0$ as follows:
\begin{equation}\label{eqn:stokesstair}
	\begin{bmatrix} M_{11} & M_{12} & 0 \\ M_{21} & M_{22} & 0 \\ 0& 0 & 0 \end{bmatrix}
	\begin{bmatrix} \dot{\hat{v}}_1 \\ \dot{\hat{v}}_2 \\  \dot{\hat{p}} \end{bmatrix} =
	\begin{bmatrix} A_{11} & A_{12} & \Sigma \\ A_{21} & A_{22} & 0 \\ -\Sigma & 0 & 0 \end{bmatrix} \begin{bmatrix} \hat v_1 \\ \hat v_2 \\ \hat p \end{bmatrix} +
	\begin{bmatrix} \hat f_1 \\ \hat f_2  \\ 0 \end{bmatrix}.
\end{equation}
}

\end{example}

\section{Obtaining and transforming the linear algebraic system}\label{sec:lin-sys}

In order to simulate the dynamical behavior of dHDAEs, time-discretization methods have to be employed. In a general (non-linear) setting this is not an easy task, since the methods have to be implicit and they should be structure preserving. Based on an ansatz derived for standard pH systems in~\cite{KotL18}, such methods were derived for dHDAE systems in~\cite{MehM19}. It was shown, in particular, that Gauss-Legendre collocation methods, like the implicit midpoint rule, are well suited for this purpose.

Here we continue to consider a linear dHDAE system of the form \eqref{eqn:dHDAE-final}. Choosing, e.g.,
a uniform time grid $t_0,\ldots,t_N$ with step size~$\tau>0$, the implicit midpoint rule yields a sequence of linear algebraic systems of the form
\begin{equation} \label{eqn:dHDAE-basic}
\Big( E+\frac{\tau}{2} (R-J)\Big) x_{k+1}= b(x_k,\tau)
\end{equation}
for the time-discrete vectors $x_{k}=x(t_{k})$, $k=0,1,2,\dots$. The linear algebraic system \eqref{eqn:dHDAE-basic} is of the form
\begin{equation}\label{eqn:dHDAE-lin-sys}
Ax=b\quad\mbox{with}\quad A=H+S,\quad\mbox{where}\quad H=E+\frac{\tau}{2}R\quad\mbox{and}\quad S=-\frac{\tau}{2}J.
\end{equation}
Thus, the splitting $A$ into its Hermitian and skew-Hermitian parts is a natural consequence of the underlying mathematical model. By construction, the Hermitian part is positive definite or positive semidefinite. Moreover, in many cases,
for any matrix norm we have $H \rightarrow E $ and $ S \rightarrow 0$ as $\tau\rightarrow 0$, so that for small step sizes~$\tau$ we can expect that the Hermitian part is dominant. This observation is of interest in the context of iteratively solving \eqref{eqn:dHDAE-lin-sys}; see Section~\ref{sec:solvers} below.

Iterative methods for solving systems of the form \eqref{eqn:dHDAE-lin-sys} are often based on the assumption that $H$ is (positive) definite and hence nonsingular; see Section~\ref{sec:solvers} below for examples.
In case of a singular matrix $H$, it is advantageous to identify its ``singular part'' and treat it separately in the numerical solution algorithm. As shown in Section~\ref{sec:examples}, the mathematical modeling frequently leads to a staircase form \eqref{stc1}--\eqref{stc2} with block matrices, where the ``singular part'' of $H$ is readily identified. If this is not possible on the modeling level, one can apply an appropriate reduction (at least in theory, or for small scale practical problems) on the algebraic level. Well-known techniques from the literature that can be applied also in this context include Schur complement constructions or null-space deflation; see, e.g.,~\cite{BenGL05,GreW19,WatG20}. We will now show how to transform $A=H+S$ using simultaneous unitary similarity transformations to a staircase form, and further to a block diagonal form, where the ``singular part'' is located in the bottom block.

The following result is a special case of the controllability staircase form~\cite{Van79}; see also~\cite{AchAM21}. We present the proof because some of its features will be used later.
\begin{lemma}\label{lem:sts}
Consider $A=H+S\in\C^{n,n}$, where $0\neq H=H^*\geq 0$ and $0\neq S=-S^*$. Then there exist a unitary matrix $U\in\C^{n,n}$, and integers $n_1 \geq n_2 \geq \cdots \geq n_{r-1} >0$ and $n_r \geq 0$, such that
\begin{equation}\label{eq:stcform}
U^*HU= \begin{bmatrix} H_{11} & 0 \\ 0 & 0\end{bmatrix}\;\;\mbox{and}\;\;
U^*SU =
\begin{bmatrix}
S_{11} & S_{12} &  & & 0 \\
S_{21} & S_{22} &  \ddots &  & 0 \\
&  \ddots & \ddots  & S_{r-2,r-1} &\vdots  \\
& &  S_{r-1,r-2}& S_{r-1,r-1} & 0\\
0 & \cdots & \cdots  & 0 & S_{r,r}
\end{bmatrix} ,
\end{equation}
where $H_{11}=H_{11}^* \in\C^{n_1,n_1}$ is positive definite, $S_{ii}=-S_{ii}^*\in\C^{n_i,n_i}$ for $i=1,\dots,r$,
and $S_{i,i-1}=-S_{i-1,i}^*=[\Sigma_{i,i-1}\;0]\in\C^{n_{i},n_{i-1}}$ with $ \Sigma_{i,i-1}$
being nonsingular for $i=2,\dots,r-1$.
\end{lemma}
\begin{proof}
The result is trivial when $H$ is nonsingular (and thus positive definite), since in this case it holds with $U=I$, $r=2$, $n_1=n$, and $n_2=0$.

Let $0\neq H=H^*\geq 0$ be singular. We consider a full rank decomposition of $H$ with a unitary matrix $U_1\in\C^{n,n}$,
\begin{equation}\label{eqn:sim-1}
U_1^*HU_1=\begin{bmatrix} \widehat{H}_{11} & 0\\ 0 & 0\end{bmatrix},
\end{equation}
where we assume that $\widehat{H}_{11}=\widehat{H}_{11}^*\in\C^{n_1,n_1}$, with $1\leq n_1<n$, is positive definite.
Note that this factorization can be obtained from any rank-revealing factorization (e.g., QR or SVD) and then applying the orthogonal factor via a congruence transformation. Applying the same unitary similarity transformation to $S$ gives the matrix
\begin{equation}\label{eqn:sim-2}
\widehat{S}=U_1^*SU_1=\begin{bmatrix} \widehat{S}_{11} & \widehat{S}_{12}\\ \widehat{S}_{21} & \widehat{S}_{22}\end{bmatrix},
\end{equation}
where $\widehat{S}_{11}\in\C^{n_1,n_1}$, and $\widehat{S}_{21}=-\widehat{S}_{12}^*$, since $S$ is skew-Hermitian. If $\widehat{S}_{21}=0$, then we are done. Otherwise, let
\[
\widehat{S}_{21}=W_2\begin{bmatrix}\Sigma_{21} & 0\\ 0 & 0\end{bmatrix}V_2^*
\]
be a singular value decomposition, where $\Sigma_{21}$ is nonsingular (and diagonal), and
$W_2\in\C^{n_1,n_1}$ and $V_2\in\C^{n-n_1,n-n_1}$ are unitary. We define $U_2={\rm diag}(V_2,W_2)\in\C^{n,n}$, which is unitary. Applying a unitary similarity transformation with this matrix to \eqref{eqn:sim-1} and \eqref{eqn:sim-2} yields
$$U_2^*U_1^*HU_1U_2=\begin{bmatrix} V_2^*\widehat{H}_{11}V_2 & 0\\ 0 & 0\end{bmatrix},$$
where $V_2^*\widehat{H}_{11}V_2\in\C^{n_1,n_1}$ is Hermitian positive definite, and
$$U_2^*U_1^*SU_1U_2=\begin{bmatrix} V_2^*\widehat{S}_{11}V_2 & V_2^*\widehat{S}_{12}W_2\\ W_2^*\widehat{S}_{21}V_2 & W_2^*\widehat{S}_{22}W_2\end{bmatrix}=
\begin{bmatrix}\widetilde{S}_{11} & \widetilde{S}_{12} & 0\\
\widetilde{S}_{21} & \widetilde{S}_{22} & \widetilde{S}_{23}\\
0 &\widetilde{S}_{32} & \widetilde{S}_{33}\end{bmatrix}$$
where $\widetilde{S}_{21}=[\Sigma_{21}\;0]$. If $\widetilde{S}_{32}=0$ or $\widetilde{S}_{32}=[\,]$, we are done. Otherwise we continue inductively with the singular value decomposition of $\widetilde{S}_{32}$, and after finitely many steps we obtain a decomposition of the required form.
\end{proof}

If for a given matrix $A=H+S$ the transformation to the staircase form \eqref{eq:stcform} is known, then the equivalent linear algebraic system $(U^*AU)(U^*x)=U^*b$ can be solved using block Gaussian elimination. This amounts to solving a sequence of linear algebraic systems having successive Schur complements as their coefficient matrices. Let us have a closer look at this process.

For simplicity of notation, we set $\widehat{A}_{11}=H_{11}+S_{11}$. By construction, this matrix is (non-Hermitian) positive definite. The set of (non-Hermitian) positive definite matrices is closed under inversion; see, e.g.,~\cite[p.~10]{Joh72}. Hence $\widehat{A}_{11}^{-1}$ exists and is also positive definite. Then in the simplest nontrivial case of the staircase form (namely, $r=3$) we can write
\begin{align*}
U^*AU &=\begin{bmatrix}\widehat{A}_{11} & S_{12} & 0\\ S_{21} & S_{22} & 0 \\ 0 & 0 & S_{33}
\end{bmatrix}\\
&=\begin{bmatrix}I & 0 & 0\\ S_{21}\widehat{A}_{11}^{-1} & I & 0 \\ 0 & 0 & I
\end{bmatrix}
\begin{bmatrix}\widehat{A}_{11} & 0 & 0\\ 0 & \mathcal{S}_1 & 0 \\ 0 & 0 & S_{33}
\end{bmatrix}
\begin{bmatrix}I & \widehat{A}_{11}^{-1} S_{12} & 0\\ 0 & I & 0 \\ 0 & 0 & I
\end{bmatrix},
\end{align*}
where $\mathcal{S}_1=S_{22}-S_{21}\widehat{A}_{11}^{-1}S_{12}$ is the Schur complement of $\widehat{A}_{11}$ in the top $2\times 2$ block. Note that the inverses of the first and third matrix in the above factorization of $U^*AU$  are obtained by simply negating the off-diagonal blocks.

Since $\widehat{A}_{11}^{-1}$ is positive definite, this matrix can be written as
\[
\widehat{A}_{11}^{-1}=\widehat{H}_{11}+\widehat{S}_{11}
\]
for some matrices  $\widehat{H}_{11}=\widehat{H}_{11}^*>0$ and $\widehat{S}_{11}=-\widehat{S}_{11}^*$. The Schur complement then is of the form
\begin{align*}
\mathcal{S}_1&=S_{22}-S_{21}\widehat{A}_{11}^{-1}S_{12}=S_{22}-S_{21}(\widehat{H}_{11}
+\widehat{S}_{11})S_{12}\\
&= (S_{21}\widehat{H}_{11}S_{21}^*)+(S_{22}+S_{21}\widehat{S}_{11}S_{21}^*),
\end{align*}
where we have used that $S_{12}=-S_{21}^*$. The Hermitian part of the Schur complement is given by
\[
S_{21}\widehat{H}_{11}S_{21}^*=\Sigma_{21}\,[I\;0]\widehat{H}_{11}\begin{bmatrix}I\\ 0\end{bmatrix}\Sigma_{21}^*.
\]
By the Cauchy interlacing theorem, the eigenvalues of $[I\;0]\widehat{H}_{11}\begin{bmatrix}I\\ 0\end{bmatrix}$ strictly interlace the eigenvalues of $\widehat{H}_{11}$. Consequently this matrix, and thus the Hermitian part and by definition $\mathcal{S}_1$ are positive definite.

Suppose that we have a further block row in the staircase form, i.e., $r=4$. Then we can write
\begin{align*}
U^*AU =\begin{bmatrix}I & 0 & 0 & 0\\ S_{21}\widehat{A}_{11}^{-1} & I & 0 & 0\\ 0 & 0 & I &0
\\ 0 & 0 & 0 & I
\end{bmatrix}
\begin{bmatrix}\widehat{A}_{11} & 0 & 0 & 0\\ 0 & \mathcal{S}_1 & S_{23} & 0\\ 0 & S_{32} & S_{33} & 0\\
0 & 0 & 0 & S_{44}
\end{bmatrix}
\begin{bmatrix}I & \widehat{A}_{11}^{-1} S_{12} & 0 & 0\\ 0 & I & 0 & 0 \\ 0 & 0 & I & 0\\ 0 & 0 & 0 & I
\end{bmatrix},
\end{align*}
where the Schur complement $\mathcal{S}_1$ is positive definite. Using the same idea as above then gives
another Schur complement $\mathcal{S}_2=S_{33}-S_{32}\mathcal{S}_1^{-1}S_{23}$, which again is positive definite. Using this block Gaussian elimination procedure inductively we obtain the following result.

\begin{lemma}\label{lem:schur}
In the notation of Lemma~\ref{lem:sts}, the matrix $U^*AU$ can be transformed via Schur complement reduction into the block diagonal form
$$\begin{bmatrix}
\widehat{A}_{11} & &  & &  \\
 & \mathcal{S}_1 &   &  &  \\
&  & \ddots  &  &  \\
& &  & \mathcal{S}_{r-2} & \\
 & &  & & S_{r,r}
\end{bmatrix},$$
where $\widehat{A}_{11}=H_{11}+S_{11}$ and the Schur complements $\mathcal{S}_1,\dots,\mathcal{S}_{r-2}$ are positive definite. Moreover, the skew-Hermitian $S_{r,r}$ may not be always present.
\end{lemma}

Lemma~\ref{lem:schur} shows that the successive formation of Schur complements leads a block diagonal matrix with all but the last block being positive definite, so that the nullspace can be obtained just from the last block.
\begin{example}\label{ex:Stokes}{\rm Consider Example~\ref{ex:stokeswostab} with
\[
E=\begin{bmatrix} M_{11} & M_{12} & 0 \\ M_{21} & M_{22} & 0 \\ 0& 0 & 0 \end{bmatrix},\quad
J=\begin{bmatrix} 0 & 0 & \Sigma \\ 0 & 0 & 0 \\ -\Sigma & 0 & 0 \end{bmatrix}, \quad
R=\begin{bmatrix} -A_{11} & -A_{12} & 0 \\ -A_{21} & -A_{22} & 0 \\ 0 & 0 & 0 \end{bmatrix}.
\]
Then $A=H+S$ (see~\eqref{eqn:dHDAE-basic}--\eqref{eqn:dHDAE-lin-sys}) is already in the staircase form (\ref{eq:stcform}) with
\[ H_{11}=\begin{bmatrix} M_{11} & M_{12}  \\ M_{21} & M_{22} \end{bmatrix}+\frac{\tau}{ 2}
\begin{bmatrix} -A_{11} & -A_{12}  \\ -A_{21} & -A_{22} \end{bmatrix},\quad
S=\begin{bmatrix}S_{11} & S_{12}\\S_{21} & S_{22}\end{bmatrix}=\left[\begin{array}{cc|c} 0 & 0 & -\frac{\tau}{2} \Sigma\\
0 & 0 & 0\\\hline \frac{\tau}{2}\Sigma & 0 & 0\end{array}\right].
\]
In the notation of Lemma~\ref{lem:sts} we have $r=3$, and $n_1\geq n_2>0=n_3$. In order to obtain the block diagonal form of Lemma~\ref{lem:schur} we have to form the Schur complement
$$\mathcal {S}_{1}= \frac{\tau^2}{ 4}[\Sigma\;\; 0]\, H_{11}^{-1}\begin{bmatrix} \Sigma \\ 0 \end{bmatrix},$$
which is positive definite.}
\end{example}

The following corollary follows immediately from Lemma~\ref{lem:schur} and the fact that the Schur complement of a skew-Hermitian matrix is again skew-Hermitian.

\begin{corollary} \label{cor:corschur}
Every Schur complement of a matrix with positive semidefinite
Hermitian part is again a matrix with this property.
\end{corollary}

In \cite{SogZ18} a similar result is shown for symmetric multiple saddle point problems in block tridiagonal form, that is to say, all consecutive Schur complements are positive definite, given that the most upper-left block is positive definite.

\section{Iterative methods for the linear algebraic systems}\label{sec:solvers}
	
In this section we discuss iterative methods for linear algebraic systems of the form $Ax=b$ with $A=H+S$.

A widely known method in this context is the HSS iteration, which was introduced in~\cite{BaiGN03}.
Given an initial vector $x^{(0)}$ and some constant $\alpha>0$, the HSS iteration successively solves linear algebraic systems with the (shifted) Hermitian and skew-Hermitian parts of $A$ by computing
\begin{align*}
(\alpha I+H) x^{(k+\frac12)} &= (\alpha I-S)x^{(k)}+b,\\
(\alpha I+S) x^{(k+1)} &= (\alpha I-H)x^{(k+\frac12)}+b,
\end{align*}
for $k=1,2,\dots$. There are numerous variants and extensions of the HSS iteration; see, e.g.,~\cite{BaiBC10,BaiG07,BaiGN03,BaiGP04,Ben09,LiYW14} or~\cite[Section~10.3]{BenGL05} for a summary of some results.
As shown in~\cite[Theorem~2.2]{BaiGN03}, the HSS iteration with exact ``inner solves'' with $\alpha I+H$ and $\alpha I+S$ converges for every $\alpha>0$, provided that $H$ (and hence $A$) is positive definite. 
However, in~\cite{BenG04} it was noted that the convergence speed of the HSS iteration is usually too slow to be competitive with other iterative methods, even when $\alpha$ is chosen optimally (in the sense that it minimizes the spectral radius of the iteration matrix).  Therefore the HSS iteration is recommended to be used as a preconditioner rather than as an iterative solver.

We will here focus on another approach, introduced in~\cite{Wid78} (also see the earlier paper~\cite{ConGol76}), which suggests to solve, instead of $Ax=b$ with $A=H+S$, the equivalent system
\begin{equation}\label{eqn:IpKsystem}
(I+K)x=\hat{b},\quad \mbox{where}\quad K=H^{-1}S,\quad \hat{b}=H^{-1}b.
\end{equation}
This transformation can be interpreted as a preconditioning of the original system with its Hermitian part, which of course requires that $H$ is nonsingular. If we again assume that $H$ is positive definite, then this matrix defines the $H$-inner product $\langle x,y\rangle_H=y^*Hx$. The adjoint of $K$ with respect to the $H$-inner product, or simply the $H$-adjoint, is given by
$$H^{-1}K^*H=H^{-1}(S^*H^{-1})H=-K,$$
hence the matrix $K$ is \emph{$H$-normal}(1), which is a necessary and sufficient condition for $K$ to admit an optimal three-term recurrence for generating an $H$-orthogonal basis of the Krylov subspaces $\mathcal{K}_k(K,v)$ for each initial vector $v$; see~\cite[Theorem~4.6.2]{LieS13}. (Note that, in addition, this implies that $K$ is diagonalizable and its eigenvalues are purely imaginary.)
This fact can be used for constructing Krylov subspace methods based on three-term recurrences for solving the system \eqref{eqn:IpKsystem}. The method of~\cite{Wid78} and a minimal residual method of~\cite{Rap78} are early and important examples. They appear to be neither widely known nor thoroughly studied, with~\cite{SzyWid93} being one of the few survey papers that discuss both methods in some detail. We will therefore summarize the most important facts about their implementation and mathematical properties here.

We first note that in matrix terms the three-term recurrence for generating an $H$-orthogonal basis of $\mathcal{K}_k(K,\hat{b})$ yields a \emph{Lanczos relation} of the form
\begin{equation}\label{eq:wid_lanczos}
KV_k=V_{k+1}T_{k+1,k},
\end{equation}
where $ {\rm Span}(V_k)=\mathcal{K}_k(K,\hat{b})$, $V_k^*HV_k=I_k$, and $T_{k+1,k}$ is tridiagonal and skew-Hermitian. Note that $V_k^*SV_k=V_k^*HV_{k+1}T_{k+1,k}=T_{k,k}$.

\subsection{Widlund's method} The method of Widlund~\cite{Wid78} is an oblique projection method with iterates $x_k^{W}$ determined by
\begin{equation*}
	x_k^W  \in \mathcal{K}_k(K,\hat{b}) \quad\text{such that}\quad
	r_k^W = b - A x_k^W  \perp \mathcal{K}_k(K,\hat{b}).
\end{equation*}
Using the Lanczos relation \eqref{eq:wid_lanczos}, we have $x_k^W = V_k y_k$ for some vector $y_k$ that is computed using the orthogonality property, i.e.,
\begin{align*}
0 &= V_k^* r_k^W = V_k^* (b - (H+S) V_ky_k) = V_k^* H \hat{b} - (V_k^* H V_k + V_k^* S V_k) y_k \nonumber\\
		&= \|{\hat{b}}\|_H e_1 - (I_k + T_{k,k}) y_k. 
\end{align*}
The system $(I_k+T_{k,k})y_k=\|{\hat{b}}\|_H e_1$ with the $k\times k$ skew-Hermitian matrix $I_k+T_{k,k}$
can be solved efficiently.

In \cite{Eise83,HagLukYou80,SzyWid93} optimality properties are shown for the even and odd subsequences $\{x_{2k}^W\}$ and $\{x_{2k+1}^W\}$, namely that
\[
\norm{x-x_{2k}^W}_H = \min_{z \in (I-K)\mathcal{K}_{2k}(K,\hat{b})} \norm{x-z}_H,
\]
and similarly for the odd subsequence. The eigenvalues of $K$ are purely imaginary. Let $i[-\lambda,\lambda]$ for some $\lambda>0$ be the smallest interval that contains these eigenvalues. Then, similar to the CG method~\cite{HesSti52}, the optimality property of Widlund's method leads to an error bound of the form
\begin{equation}\label{eqn:CGW-bound}
\frac{\norm{x-x_{2k}^W}_H}{\norm{x}_H}\leq 2 \left(\frac{\sqrt{1+\lambda^2}-1}{\sqrt{1+\lambda^2}+1}\right)^k,
\end{equation}
and the same bound holds for the sequence $\norm{x-x_{2k+1}^W}_H/\norm{x-x_{1}^W}_H$; see~\cite{Eise83} or~\cite[Theorem~4.2]{SzyWid93}. The bound indicates that a ``fast'' convergence of the method can be expected when $\lambda>0$ is ``small''.
	
\subsection{Rapoport's method} The method of Rapoport~\cite{Rap78} is a minimal residual method with iterates $x_k^R$ determined by
\begin{equation} \label{rap_ort}
x_k^R \in \mathcal{K}_k(K,\hat{b}) \quad\text{such that}\quad
r_k^R = b - A x_k^R \perp (I+K)\mathcal{K}_k(K,\hat{b}).
\end{equation}
Since the Lanczos relation \eqref{eq:wid_lanczos} can be written as
\[
(I+K)V_k = V_{k+1}\begin{bmatrix}
		I_k+T_{k,k} \\ t_{k+1,k}e_k^*
	\end{bmatrix}\equiv V_{k+1}\tilde{T}_{k+1,k},
\]
we obtain $x_k^R = V_k y_k$ for some vector $y_k$ determined by the orthogonality property, i.e.,
\begin{align*}
		0 &= ((I+K)V_k)^* r_k^R = \tilde{T}_{k+1,k}^*V_{k+1}^* H (\hat{b} - (I+K) V_ky_k) \\
		&= \|{\hat{b}}\|_H ~\tilde{T}_{k+1,k}^* e_1 -  \tilde{T}_{k+1,k}^*  \tilde{T}_{k+1,k} y_k.
\end{align*}
Equivalently, $y_k$ is the solution of the least squares problem
$$\min_y \|\|\hat{b}\|_H e_1-\tilde{T}_{k+1,k}y\|_2,$$
which can again be solved efficiently, since $\tilde{T}_{k+1,k}$ is tridiagonal.
	
Since $A=H(I+K)$, we have $r_k^R=H(I+K)(x-x_k^R)$, and we can write the orthogonality property
in \eqref{rap_ort} as
\[
x-x_k^R \perp_{B} \mathcal{K}_k(I+K,\hat{b}) =\mathcal{K}_k(K,\hat{b}),
\]
where $B\equiv (I+K)^*H(I+K)$ is Hermitian positive definite. Since $x_k^R \in \mathcal{K}_k(K,\hat{b})$,
this is mathematically equivalent to the optimality property
\[
\norm{x-x_k^R}_{B}=\min_{z\in \mathcal{K}_k(K,\hat{b})}\norm{x-z}_{B};
\]
see~\cite[Theorem~2.3.2]{LieS13}. We thus obtain
\begin{align}
		\norm{b-Ax_k^R}_{H^{-1}} &= \norm{H^{-1}(b-Ax_k^R)}_H =\norm{(I+K)(x-x_k^R)}_H \nonumber\\
		&= \norm{x-x_k^R}_{B}= \min_{z\in \mathcal{K}_k(K,\hat{b})}\norm{x-z}_{B}\nonumber\\
		&= \min_{z\in \mathcal{K}_k(K,\hat{b})}\|\hat{b}-(I+K)z\|_{H}
        =\min_{z\in \mathcal{K}_k(I+K,\hat{b})}\|\hat{b}-(I+K)z\|_{H}\nonumber\\
		&=\min_{p(0)=1\atop {\rm deg}(p)\leq k}\|p(I+K)\hat{b}\|_{H}=
        \min_{p(0)=1\atop {\rm deg}(p)\leq k}\|H^{-1}p(AH^{-1})b\|_{H}\label{eqn:R-bd}\\
        &=\min_{p(0)=1\atop {\rm deg}(p)\leq k}\|p(AH^{-1})b\|_{H^{-1}},\nonumber
\end{align}
where we have used that $\mathcal{K}_k(K,\hat{b})=\mathcal{K}_k(I+K,\hat{b})$. The matrix $I+K$ is $H$-normal(1), and hence diagonalizable with an $H$-unitary matrix of eigenvectors, i.e., $I+K=Y\Lambda Y^{-1}$ and $Y^*HY=I$. Note that $H^{1/2}Y$ is unitary. Using the first expression in \eqref{eqn:R-bd} we thus obtain
\begin{align*}
\norm{b-Ax_k^R}_{H^{-1}} &= \min_{p(0)=1\atop {\rm deg}(p)\leq k}\|p(I+K)\hat{b}\|_{H}\\
&=\min_{p(0)=1\atop {\rm deg}(p)\leq k}\|H^{1/2}Yp(\Lambda)Y^{-1}H^{-1}b\|_{2}\\
&\leq \|(Y^{-1}H^{-1/2})H^{-1/2} b\|_{2}\,\min_{p(0)=1\atop {\rm deg}(p)\leq k}\|H^{1/2}Yp(\Lambda)\|_2\\
&=\|b\|_{H^{-1}}\,\min_{p(0)=1\atop {\rm deg}(p)\leq k}\|p(\Lambda)\|_2.
\end{align*}
The polynomial minimization problem on the spectrum of $I+K$, which is contained in a complex interval of the form $1+i[-\lambda,\lambda]$ for some $\lambda>0$, was considered in~\cite{FreRus86} (see also~\cite{Fre90}), and it leads to a convergence bound of the form
\begin{equation}\label{eqn:Rap-bd}
\frac{\norm{b-Ax_k^R}_{H^{-1}}}{\|b\|_{H^{-1}}}\leq 2\left(\frac{\lambda}{\sqrt{1+\lambda^2}+1}\right)^k;
\end{equation}
see~\cite[Theorem~4.3]{SzyWid93}. As for Widlund's method, this bound for Rapoport's method indicates that the convergence is ``fast'' when $\lambda>0$ is ``small''.

\subsection{Comparison with GMRES} We will now compare the methods of Widlund and Rapoport with GMRES~\cite{SaaSch86}. Recall that the GMRES method applied to $Ax=b$ and starting with $x_0=0$ has iterates $x_k^G$ that are determined by
\begin{equation*} 
x_k^G\in \mathcal{K}_k(A,b)\quad\mbox{such that}\quad r_k^{G}=b-Ax_k^G \perp A\mathcal{K}_k(A,b), \end{equation*}
and that the orthogonality property of the method is equivalent to the optimality property
\begin{equation*} 
\norm{r_k^{G}}_2 = \min_{z \in \mathcal{K}_k(A,b)} \norm{b-Az}_2=
\min_{p(0)=1\atop {\rm deg}(p)\leq k}\norm{p(A)b}_2.
\end{equation*}
Note that the GMRES method is well defined when $A$ is nonsingular, but in contrast to the methods of Widlund and Rapoport it is based on full rather than three-term recurrences.

Analogously, an application of GMRES with $x_0=0$ to the left-preconditioned system $H^{-1}Ax=\hat{b}$ has iterates $x_k^{LG}$ that are characterized by
\begin{equation*}
x_k^{LG}
\in {\mathcal{K}}_k(H^{-1}A,\hat{b}) \quad \mbox{such that} \quad r_k^{LG}=\hat{b}-H^{-1}Ax_k^{LG}\perp H^{-1}A{\mathcal{K}}_k(H^{-1}A,\hat{b}).
\end{equation*}
This method is well defined when $H$ is nonsingular, but $H$ does not need to be definite.
Note that $\mathcal{K}_k(K,\hat{b})=\mathcal{K}_k(I+K,\hat{b})=\mathcal{K}_k(H^{-1}A,\hat{b})$, and hence GMRES applied to the left-preconditioned system uses the same search spaces for the iterates as the methods of Widlund and Rapoport. The optimality property now is given by
\begin{align}
\norm{r_k^{LG}}_2 &= \min_{z \in \mathcal{K}_k(H^{-1}A,\hat{b})} \|\hat{b}-H^{-1}Az\|_2=
\min_{p(0)=1\atop {\rm deg}(p)\leq k} \|p(H^{-1}A)\hat{b}\|_2 \label{eq:lgmres_opt}\\
&=\min_{p(0)=1\atop {\rm deg}(p)\leq k} \|H^{-1}p(AH^{-1})b\|_2
=\min_{p(0)=1\atop {\rm deg}(p)\leq k} \|p(AH^{-1})b\|_{H^{-2}}.\nonumber
\end{align}
If we again write $I+K=Y\Lambda Y^{-1}$, then the last expression in \eqref{eq:lgmres_opt} leads to
the bound
\begin{equation}\label{eqn:L-GMRES-bd}
\frac{\norm{r_k^{LG}}_2}{\|\hat{b}\|_2}\leq \kappa(Y) \min_{p(0)=1\atop {\rm deg}(p)\leq k} \|p(\Lambda)\|_2,
\end{equation}
which reminds one of the standard GMRES convergence bound for diagonalizable matrices, and where the minimization problem can be bounded as in \eqref{eqn:Rap-bd}. Moreover, we have
\begin{align*}
\|r_k^{LG}\|_2=\|\hat{b}-H^{-1}Ax_k^{LG}\|_2=\|H^{-1}(b-Ax_k^{LG})\|_2=
\|b-Ax_k^{LG}\|_{H^{-2}},
\end{align*}
where $b-Ax_k^{LG}$ can be considered the \emph{unpreconditioned residual} of the GMRES method applied to the left-preconditioned system.

Table~\ref{tabsum} contains an overview of the mathematical characterizations and optimality properties of the four methods discussed above, where L-GMRES means GMRES applied to the left preconditioned system.

We point out that several Krylov subspace methods with short recurrences have been proposed in the literature for the solution of linear algebraic systems with (shifted) skew-Hermitian or skew-symmetric matrices; see, e.g., the survey~\cite{SzyWid93} or the more recent papers~\cite{GrePaiTitVar16,GreVar09,IdeV07,Jia07}. Here we do not take these methods into account, since our matrix~$K$ is skew-Hermitian with respect to the $H$-inner product, so that methods for usual skew-Hermitian matrices are not directly applicable. Moreover, the methods of Widlund and Rapoport already implement the two most common projection principles in this context, namely oblique and orthogonal projection onto Krylov subspaces.

\begin{table}
\renewcommand{\arraystretch}{1.3}
\begin{tabular}{llll}
\multicolumn{4}{l}{\textbf{Mathematical characterization:}}\\
Widlund:	& $x_k^W \in \mathcal{K}_k(K,\hat{b}) $ & such that $r_k^W= b - A x_k^W  $ & $\perp  {\mathcal{K}}_k(K,\hat{b})$  \\
Rapoport: & $x_k^R \in {\mathcal{K}}_k(K,\hat{b}) $ & such that $r_k^R = b - A x_k^R $ & $\perp  H^{-1}A\mathcal{K}_k(K,\hat{b})$  \\
L-GMRES:	&    $x_k^{LG} \in \mathcal{K}_k(K,\hat{b}) $ & such that $ r_k^{LG}=\hat{b}-H^{-1}Ax_k^{LG} $ & $\perp  H^{-1}A\mathcal{K}_k(K,\hat{b})$  \\
GMRES:	&  $x_k^G \in \mathcal{K}_k(A,b) $ & such that $r_k^G= b - A x_k^G  $ & $\perp  A\mathcal{K}_k(A,b)$  \\[1ex]
\multicolumn{4}{l}{\textbf{Minimization properties:}}\\
Widlund:		 &  $\norm{x-x_{2k}^W}_H $ & $= \displaystyle \min_{z \in (I-K)\mathcal{K}_{2k}(K,\hat{b})} \norm{x-z}_H$   &    \\
Rapoport:	 &   $\norm{b-Ax_k^R}_{H^{-1}} $ & $= \displaystyle \min_{z \in \mathcal{K}_k(K,\hat{b})} \norm{b-Az}_{H^{-1}}  $ &   \hspace*{-1.3cm}  $=\displaystyle \min_{p(0)=1\atop {\rm deg}(p)\leq k}\norm{p(AH^{-1})b}_{H^{-1}}$ \\
L-GMRES: &    $\norm{b-Ax_k^{LG}}_{H^{-2}} $ & $= \displaystyle
\min_{z \in \mathcal{K}_k(K,\hat{b})} \|b-Az\|_{H^{-2}}$ & \hspace*{-1.3cm} $= \displaystyle\min_{p(0)=1\atop {\rm deg}(p)\leq k} \|p(AH^{-1})b\|_{H^{-2}}$\\
GMRES:	&   $\norm{b-Ax_k^G}_2 $ & $= \displaystyle \min_{z \in \mathcal{K}_k(A,b)} \norm{b-Az}_2 $  &  \hspace*{-1.3cm}
$= \displaystyle \min_{p(0)=1\atop {\rm deg}(p)\leq k}\norm{p(A)b}_2 $
\end{tabular}
\caption{Mathematical characterization and minimization properties of the different methods.}\label{tabsum}
\end{table}

\section{Numerical experiments}	\label{sec:tests}

In this section we present numerical experiments with the four iterative methods summarized in Table~\ref{tabsum} applied to linear algebraic systems of the form \eqref{eqn:dHDAE-lin-sys}, which come from different cases discussed in Section~\ref{sec:examples}. All experiments were carried out in MATLAB R2019b on a cluster with an AMD EPYC 7302 16-Core Processor and 512GB memory.

We have implemented the methods of Widlund and Rapoport in MATLAB as stated in~\cite{Wid78} and~\cite{Rap78}, respectively, and we use the MATLAB implementation of (preconditioned) GMRES. The methods of Widlund and Rapoport as well as L-GMRES are based on preconditioning the system~\eqref{eqn:dHDAE-lin-sys} with the Hermitian part $H$ of $A$, and hence they require solving a linear algebraic system with $H$ in every step. In large-scale problems one can compute a Cholesky decomposition of $H$, and then solve the two triangular systems in every step. We point out that for a sequence of linear algebraic systems coming from a discretization with constant time steps as in \eqref{eqn:dHDAE-basic}, only one Cholesky decomposition needs to be computed upfront. We will comment on our use of the Choelsky decomposition in the different examples below.

As seen in Table~\ref{tabsum}, the four methods minimize different norms of residual or error. We consider GMRES applied to the non-preconditioned system $Ax=b$ as the reference method. This method minimizes the 2-norm of the residual in every step, and therefore we compare the residuals of all four methods in this norm. In all experiments we start the iterative methods with $x_0=0$, and run the iterations until the relative residual 2-norm is smaller than a given tolerance.

\subsection{Multi-body system (Case 1)}\label{num:Mbs}

We consider the holonomically constrained damped mass-spring system illustrated
in Figure \ref{fig:mbs} (taken from~\cite[Fig.~3.4]{MehS05}). The $i$th mass of weight $m_i$ is connected to the $(i + 1)$st mass by a spring and a damper with constants $k_i$ and $d_i$, respectively, and also to the ground by a spring and a damper with constants $\kappa_i$ and $\delta_i$ respectively. Additionally, the first mass is connected to the last one by a rigid bar and it is influenced by a control.

\begin{figure}[h]
	\centering
	\includegraphics[width=0.7\linewidth]{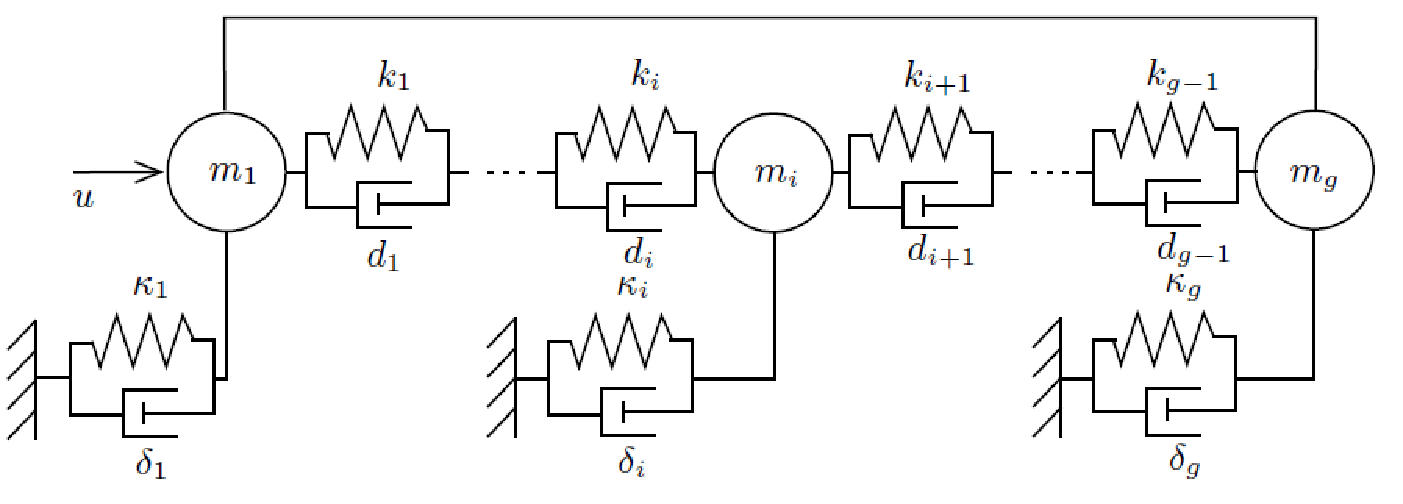}
	\caption{A damped mass-spring system with a holonomic constraint.}
	\label{fig:mbs}
\end{figure}

The vibration of this system is described by a descriptor system (see~\cite[equation~(34)]{MehS05}), from which one obtains an equation of the form
$$M\ddot{p}+D\dot{p}+Fp=f,$$
where $M=\diag(m_1, \dots, m_g)>0$ is the mass matrix, and $D=D^*\geq 0$ and $F=F^*>0$ are the tridiagonal damping and stiffness matrices, respectively. The resulting first-order formulation gives the state equation of a dHDAE system of the form \eqref{eq_brake_dae}. After time discretization we obtain a linear algebraic system of the form \eqref{eqn:dHDAE-lin-sys} with $A$ of order $n=2\cdot g$, and the positive definite Hermitian part
$$H=E+\frac{\tau}{2} R=\begin{bmatrix} M +\frac{\tau}{2} D& 0\\0 & F\end{bmatrix}.$$

We consider linear algebraic systems for two different values of $g$, namely $g=5 \times 10^{3}$ and $g=10^{6}$, and in each case we use time steps of four different orders of magnitude, namely $\tau/2=10^{-4},\,10^{-3},\,10^{-2},\,10^{-1}$. The right hand sides of the linear algebraic systems are generated by the command {\tt randn} in MATLAB.  In each case we compute one Cholesky decomposition of $H$ using MATLAB's {\tt chol} function, and then solve the two triangular systems in every iterative step of the methods of Widlund and Rapoport, and L-GMRES with MATLAB's backslash operator. The time for computing the Cholesky decomposition is $0.0009$s and $0.2070$s for $g=5 \times 10^{3}$ and $g=10^{6}$, respectively, and is included in the running times of the methods of Widlund, Rapoport, and L-GMRES shown in Tables~\ref{table:Mbs-small} and~\ref{table:Mbs-large}. In the computations with the iterative methods the tolerance for the relative residual norm is $10^{-12}$. In Figures~\ref{fig:Mbs-small} and~\ref{fig:Mbs-large} we plot the convergence curves of the four iterative methods for $g=5 \times 10^{3}$ and $g=10^6$, respectively.

The tables and figures show that GMRES is outperformed in terms of time and iterative steps by the methods of Widlund, Rapoport and L-GMRES. This is not surprising, since the latter three methods are all preconditioned by the dominant Hermitian part. Moreover, in each case these three methods take approximately the same number of steps to reach the stopping criterion. Because of the full recurrences, L-GMRES in each case takes a (slightly) longer time than the methods of Widlund and Rapoport.

We also observe that the number of steps required by the methods of Widlund, Rapoport, and L-GMRES to reach the stopping criterion increases with increasing $\tau$. A heuristic explanation of this observation is that with increasing $\tau$ the Hermitian part becomes ``less dominant'', so that the gain of using it as a preconditioner, and hence the advantage over (unpreconditioned) GMRES, becomes less pronounced. A more analytic (though still not complete) explanation is given by the convergence bounds \eqref{eqn:CGW-bound}, \eqref{eqn:Rap-bd}, and \eqref{eqn:L-GMRES-bd}. For $g=5\times 10^3$ the smallest (purely imaginary) interval $i[-\lambda,\lambda]$ containing these eigenvalues is given by
\begin{align*}
\mbox{$\lambda\approx 3.1622 \times 10^{-5}$ for $ \tau/2=10^{-4}$,}\\
\mbox{$\lambda\approx 3.1619 \times 10^{-4}$ for $\tau/2=10^{-3}$,}\\
\mbox{$\lambda\approx 3.1583 \times 10^{-3}$ for $ \tau/2=10^{-2}$,}\\
\mbox{$\lambda\approx 3.1235 \times 10^{-2}$ for $ \tau/2=10^{-1}$.}
\end{align*}
Thus, decreasing $\tau$ by a factor of 10 means that the spectrum of $K$ ``shrinks'' by the same factor. The faster convergence for smaller $\tau$ is indicated by the convergence bounds, which say that the convergence of the three methods is ``fast'' when $\lambda>0$ is ``small''. (The same behavior and conclusion hold for $g=10^6$, but in that case we did not compute the eigenvalues of the matrix~$K$.)

\begin{table}
	\begin{tabular*}{\textwidth}{@{\extracolsep{\fill}}l rr rr rr rr}
		& \multicolumn{2}{c}{$\tau/2=10^{-4}$} & \multicolumn{2}{c}{$\tau/2=10^{-3}$} & \multicolumn{2}{c}{$\tau/2=10^{-2}$} & \multicolumn{2}{c}{$ \tau/2=10^{-1}$}\\				
		\hline
		Method       &  	 Time  & Iter.	 &  	 Time  & Iter.	& Time  & Iter.  & Time  & Iter. \\
		\hline
		Widlund	&  0.003 	&    3   &    0.005 	&    4 & 0.004 & 5& 0.003 & 7\\
		Rapoport &  0.002 	 &  2 	  & 0.004 	 &   3  & 0.008 & 4 & 0.005 & 6\\
		L-GMRES	&  0.028 &	    3	 &  0.014 	 &   4 & 0.016 & 5 & 0.010 &7\\
		GMRES	& 0.032 	&	   35  &   0.029 	 &  37 & 0.088 & 39 & 0.030 & 44
		
	\end{tabular*}
	\centering
	\caption{Multi-body system. Running times and iteration numbers for $g=5 \times 10^{3}$.} \label{table:Mbs-small}
\end{table}

\begin{table}
	\begin{tabular*}{\textwidth}{@{\extracolsep{\fill}}l rr rr rr rr}
		& \multicolumn{2}{c}{$\tau/2=10^{-4}$} & \multicolumn{2}{c}{$\tau/2=10^{-3}$} & \multicolumn{2}{c}{$\tau/2=10^{-2}$} & \multicolumn{2}{c}{$ \tau/2=10^{-1}$}\\				
		\hline
		Method       &  	 Time  & Iter.	 &  	 Time  & Iter.	&  	 Time  & Iter. &  	 Time  & Iter. \\
		\hline
		Widlund	&  	  0.458 & 3 &  0.533 &   4 & 0.687 & 6 & 0.930 & 9\\
		Rapoport & 0.334 	&   2 &   0.424 &   3  & 0.619 & 5 & 0.939 & 8\\
		L-GMRES & 0.720 & 3&	0.837 	&   4& 1.099 & 6 & 1.544 & 9\\
		GMRES	&  4.839 &   28 &   4.734 	&	28& 5.553 & 29 & 5.525 & 31
	\end{tabular*}
	\centering
	\caption{Multi-body system. Running times and iteration numbers for  $g=10^6$.} \label{table:Mbs-large}
\end{table}

\begin{figure}
	\centering
	\includegraphics[width=0.49\textwidth]{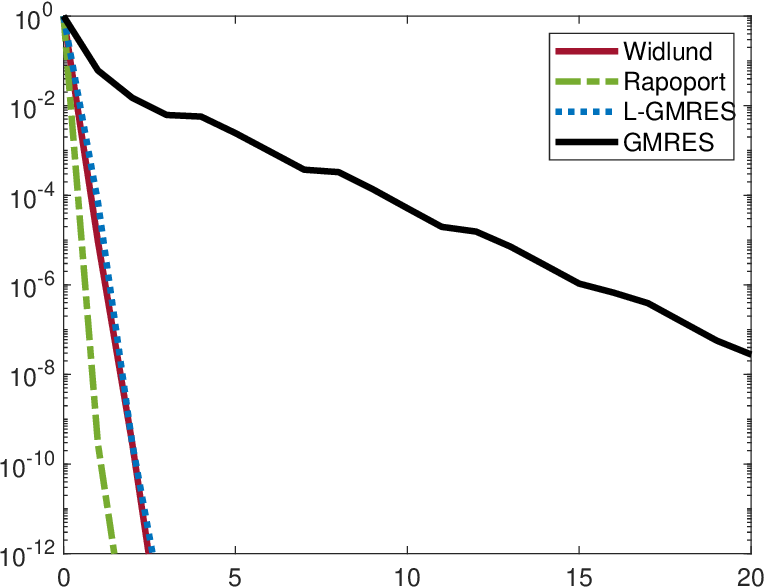}
	\includegraphics[width=0.49\textwidth]{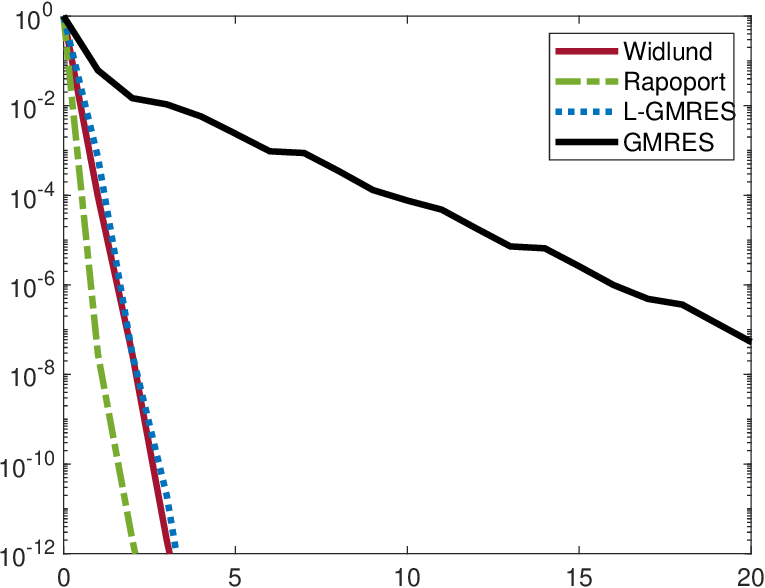}
	\includegraphics[width=0.49\textwidth]{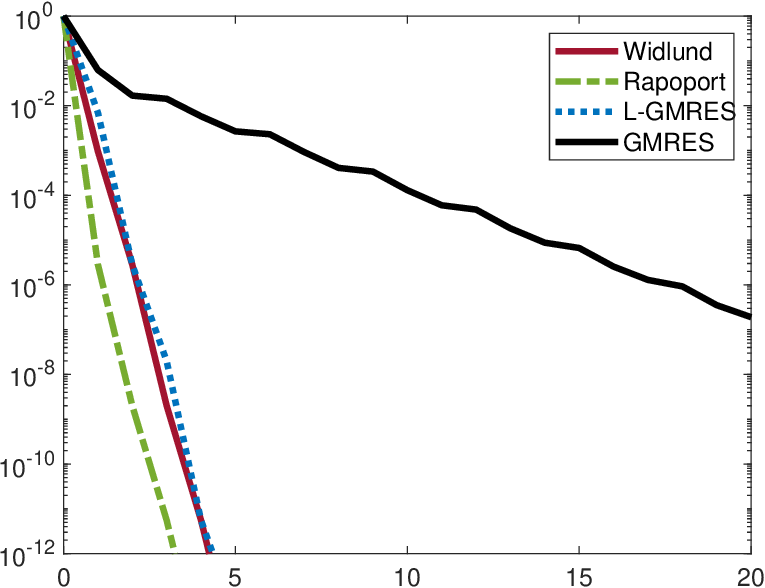}
	\includegraphics[width=0.49\textwidth]{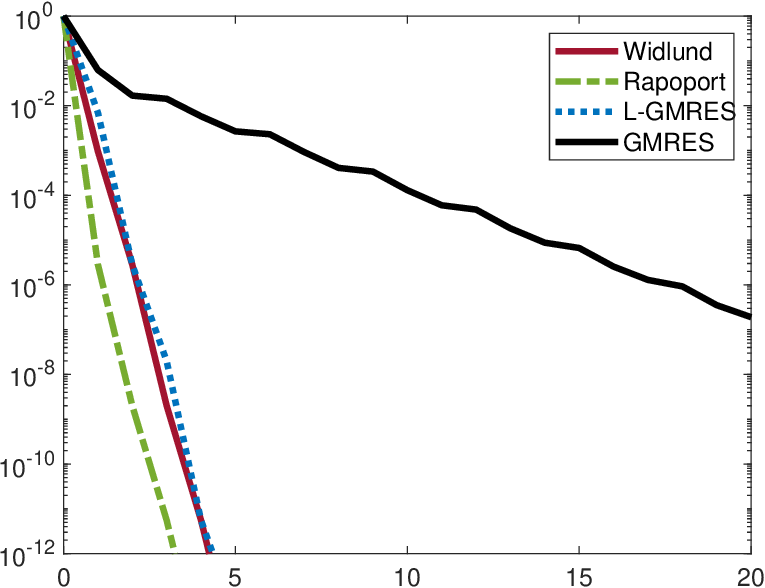}
	\caption{Multi-body system. Relative residual norms of the four methods with
    $ \tau/2=10^{-4},\,10^{-3},\,10^{-2},\,10^{-1}$ (top left to bottom right) for $g=5 \times 10^{3}$.}
	\label{fig:Mbs-small}
\end{figure}

\begin{figure}
	\centering
	\includegraphics[width=0.49\textwidth]{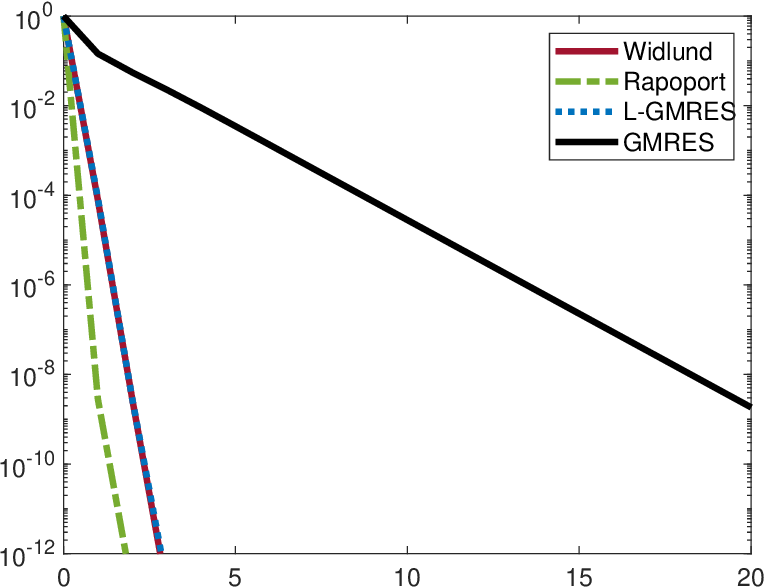}
	\includegraphics[width=0.49\textwidth]{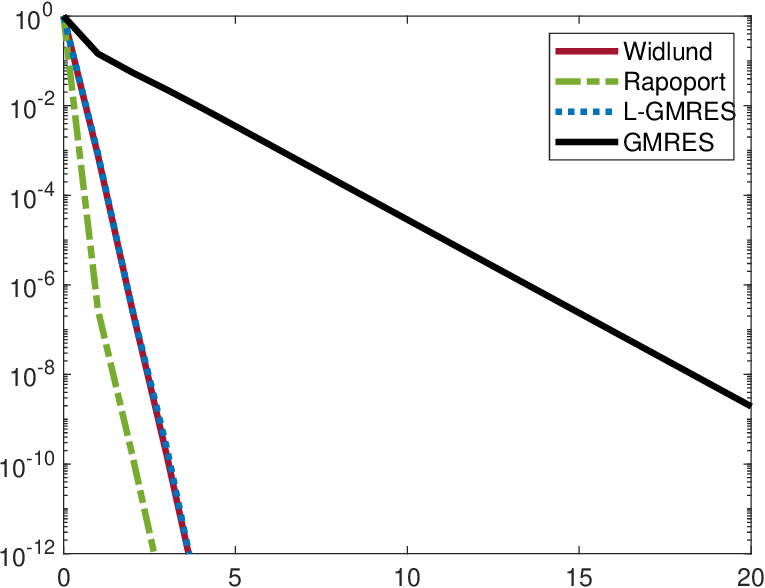}
	\includegraphics[width=0.49\textwidth]{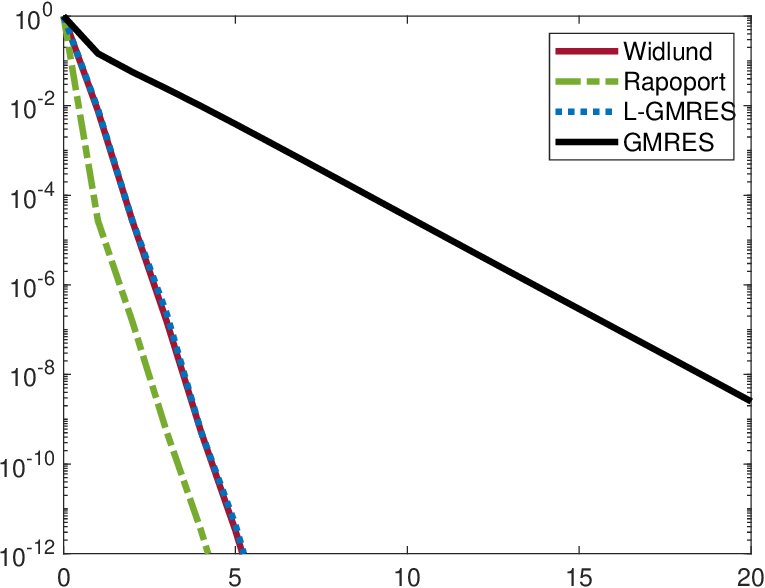}
	\includegraphics[width=0.49\textwidth]{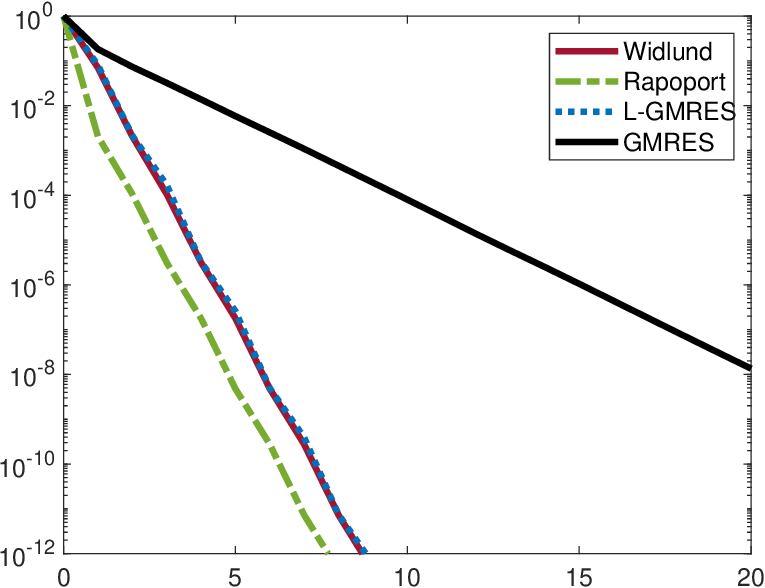}
	\caption{Multi-body system. Relative residual norms of the four methods with
    $\tau/2=10^{-4},\,10^{-3},\,10^{-2},\,10^{-1}$ (top left to bottom right) for $g=10^{6}$.}
	\label{fig:Mbs-large}
\end{figure}
%


\subsection{Stokes equation (Case 2)}\label{num:Stokes}
We consider the incompressible Stokes equation as in Example~\ref{ex:stokes}, and generate linear algebraic systems using the $Q1-Q1$ finite element approximation of the (unsteady) channel domain problem in IFISS~\cite{ifiss}. The system matrices are of the form $A=H+S$ with
$$H=\begin{bmatrix} M-\frac{\tau}2 A_H& 0\\ 0 & -\frac{\tau}2 C\end{bmatrix}\quad\mbox{and}\quad
S=\begin{bmatrix}  0& -\frac{\tau}2 B\\ \frac{\tau}2 B^* & 0\end{bmatrix},$$
where $M-\frac{\tau}2 A_H$ with $A_H={\rm diag}(A_1,A_1)$ is Hermitian positive definite. 
We use the stabilization matrix
$-\frac{\tau}{2}C=10^{-3}\frac{\tau}{2}A_1$, so that the methods of Widlund and Rapoport are applicable. We use the grid parameters 6 and 9 in IFISS, which yields matrices $B$ of the sizes $8,450\times 4,225$ and $526,338\times 263,169$, respectively. Hence, for the matrices $A$ we have $n=12,675$ and $n=789,507$, respectively. The right hand sides $f$ of the linear algebraic systems are generated by the command {\tt randn} in MATLAB.

We consider linear algebraic systems corresponding to the time steps $\tau/2=10^{-4}$ and 
$\tau/2= 10^{-3}$.
In each case we compute Cholesky decompositions of $M-\frac{\tau}2 A_H$ and $-\frac{\tau}{2}C$ using MATLAB's {\tt chol} function. For grid parameter 6, this takes about $0.01$s for both $M-\frac{\tau}2 A_H$ and $-\frac{\tau}{2}C$, and both values of $\tau/2$. For grid parameter 9, the computation of the Cholesky factors of $M-\frac{\tau}2 A_H$ and $-\frac{\tau}{2}C$ respectively takes
$7.22$s and $9.58$s for $\tau/2=10^{-4}$, as well as $6.99$s and $12.32$s for $\tau/2=10^{-3}$.
The triangular systems with the Cholesky factors are then solved in every iterative step of the methods of Widlund and Rapoport, and L-GMRES with MATLAB's backslash operator.
The time for computing the Cholesky decompositions are included in the running times of the methods of Widlund, Rapoport, and L-GMRES shown in Tables~\ref{table:Stokes-small} and~\ref{table:Stokes-large}.  Figures~\ref{fig:Stokes-small} and~\ref{fig:Stokes-large} show the corresponding convergence curves of the four iterative methods. In the computations with the iterative methods the tolerance for the relative residual norm is $10^{-12}$. We stopped the (unpreconditioned) GMRES method after 1000 steps, and we report the value of the relative residual norm attained at that point.

The observations in this example are similar to those for the multibody system in Section~\ref{num:Mbs}. The methods of Widlund, Rapoport, and L-GMRES behave similarly in terms of iterative steps, and L-GMRES takes a slightly longer time (in most cases) due to the full recurrences. The number of steps of each of these three methods increases with increasing $\tau$.
The smallest interval $i[-\lambda,\lambda]$ containing the eigenvalues of $K$ is given by
\begin{equation*}
\lambda\approx 2.158 \text{ for } \frac{\tau}{2}=10^{-4} \quad \text{ and } \quad \lambda \approx 4.336 \text{ for } \frac{\tau}{2}=10^{-3}
\end{equation*}
%
and hence the convergence bounds for the methods of Widlund, Rapoport, and L-GMRES again explain (to some extent) why the methods require fewer steps for smaller~$\tau$. Finally, it is noteworthy that the residual norms of the method of Widlund show rather large oscillations, while the method of Rapoport, which minimizes the $H^{-1}$-norm of the residual, converges smoothly.

\begin{table}

\begin{tabular*}{\textwidth}{@{\extracolsep{\fill}}l rrr rrr}
		& \multicolumn{3}{c}{$\tau/2=10^{-4}$} & \multicolumn{3}{c}{$\tau/2=10^{-3}$} \\		
		\hline
		Method       &  	 Time & $\norm{Rel. Res.}$ 	 &   Iter. &	 Time  &  $\norm{Rel. Res.}$  & Iter. 	 \\	
		\hline
		Widlund     &    0.095& $ 6.515\times 10^{-14}$ &  	   33  &  0.121 &$	 1.816\times 10^{-13}$ &  	   55 	 \\
		Rapoport    &   0.143	& $ 2.819\times 10^{-13} $ & 	   33 &   0.128 	 & $7.707 \times 10^{-13}  	$ &   53   \\
		L-GMRES &     0.140 & $	 5.938\times 10^{-13}  	$ &   32 &   0.193 & $	 3.021\times 10^{-13} $ & 	   50  	 \\
		GMRES    &       14.174 	 & $3.066\times 10^{-09}  $ &	  1000  &     15.129 	 & $9.208\times 10^{-13}  	 $ & 996
	\end{tabular*}

	\centering
	\caption{Stokes equation. Running times and iteration numbers for $n=12,675$.}	
	\label{table:Stokes-small}
\end{table}

\begin{table}

\begin{tabular*}{\textwidth}{@{\extracolsep{\fill}}l rrr rrr}
		& \multicolumn{3}{c}{$\tau/2=10^{-4}$} & \multicolumn{3}{c}{$\tau/2=10^{-3}$} \\		
		\hline
		Method       &  	 Time & $\norm{Rel. Res.}$ 	 &   Iter. &	 Time  &  $\norm{Rel. Res.}$  & Iter. 	 \\	
		\hline
		Widlund     & 42.852 & $	 1.476\times 10^{-13}  	 $ &  33  &   56.062 	& $ 2.264\times 10^{-13}  	 $ &  53 	 \\
		Rapoport    &    40.711	& $ 2.614\times 10^{-13} $ & 	   35  &   56.777  	& $ 4.293\times 10^{-13}  $ &	   55   \\
		L-GMRES & 43.831 	 & $6.587\times 10^{-13}  $ &	   34  & 57.757& $	 8.241\times 10^{-13}  	$ &   50 	 \\
		GMRES    &  983.591 & $	5.762\times 10^{-06}  	$ &  1000 &  980.367	 & $7.179\times 10^{-05}  	$ &  1000
	\end{tabular*}
	\centering
	\caption{Stokes equation. Running times and iteration numbers for $n= 789,507$.}	
	\label{table:Stokes-large}
\end{table}

\begin{figure}
\centering
\includegraphics[width=0.49\textwidth]{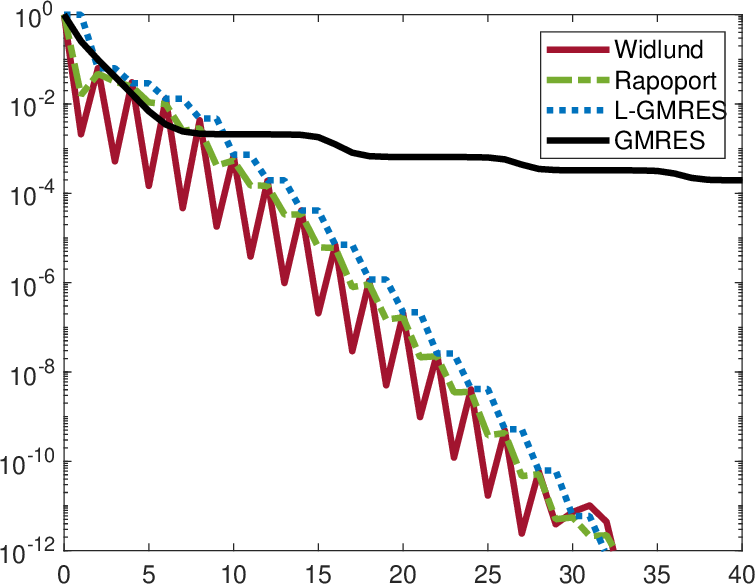}
\includegraphics[width=0.49\textwidth]{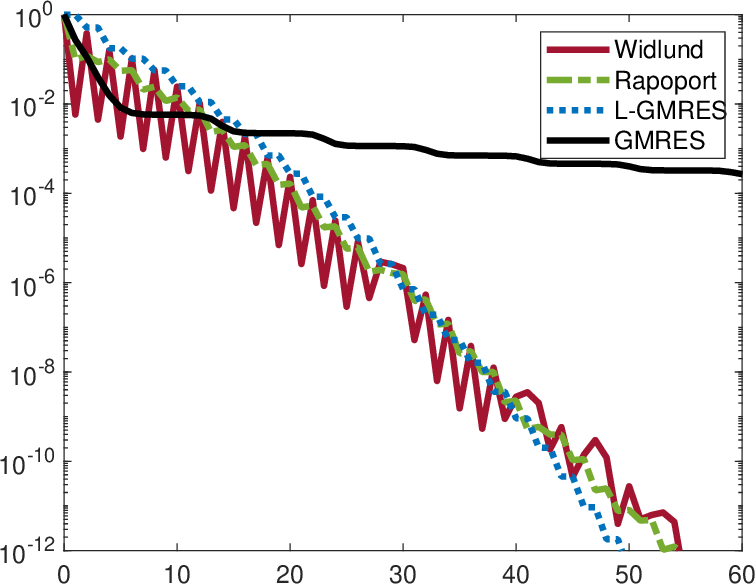}
\caption{Stokes equation. Relative residual norms of the four methods with $\tau/2=10^{-4}$ and $\tau/2=10^{-3}$ (left and right) for $n= 12,675$.}	
\label{fig:Stokes-small}
\end{figure}

\begin{figure}
	\centering
	\includegraphics[width=0.49\textwidth]{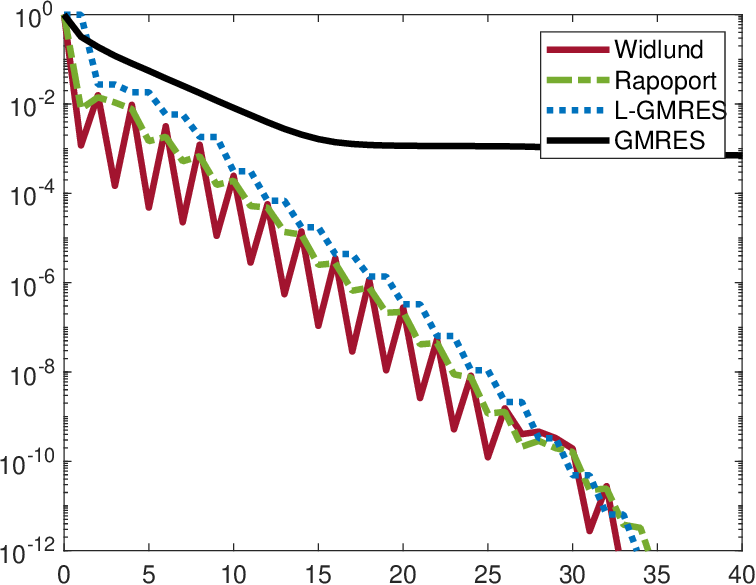}
	\includegraphics[width=0.49\textwidth]{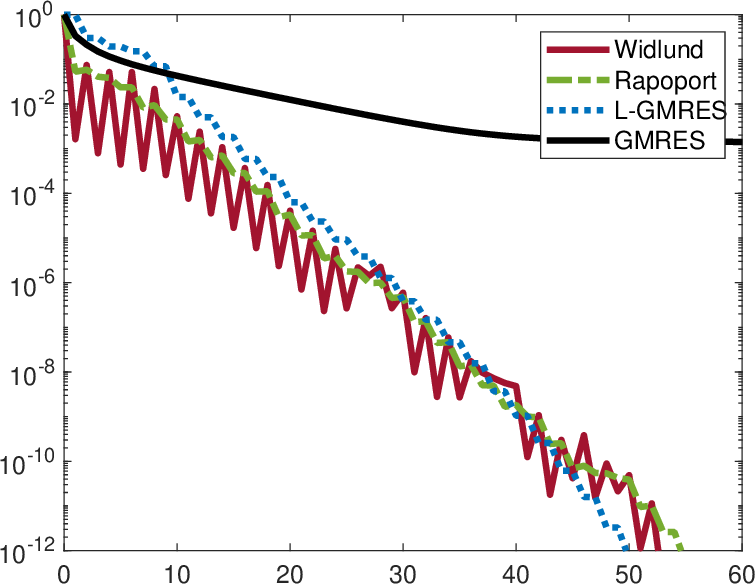}
	\caption{Stokes equation. Relative residual norms of the four methods with $\tau/2=10^{-4}$ and $\tau/2=10^{-3}$ (left and right) for $n= 789,507$.}	
	\label{fig:Stokes-large}
\end{figure}

\subsection{Linearized Navier-Stokes equation without stabilization (Case 3)}\label{num:Navier}

As a final example we consider the linearized Navier-Stokes equation without stabilization as in Example \ref{ex:stokeswostab}, and generate linear algebraic systems using the $Q2-Q1$ finite element discretization of the (unsteady) channel domain problem in IFISS~\cite{ifiss}. Now the systems are of the form
\begin{equation}\label{eqn:lin_nav}
Ax= \begin{bmatrix}M -\frac{\tau}2 (A_H+A_S)& \frac{\tau}2 B\\ -\frac{\tau}2 B^* & 0\end{bmatrix} \begin{bmatrix} v\\p\end{bmatrix} = \begin{bmatrix}f\\0\end{bmatrix}=b,
\end{equation}
where $M-\frac{\tau}{2} A_H$ is Hermitian positive definite, $A_S=-A_S^*$, and $B$ is of full rank. We use the ``grid parameter'' 6, so that $B$ is of size $8,450\times 1,089$, and hence $n=9,539$. The vector~$f$ for the right hand side is also generated by IFISS. (In this example we only use a relatively small value of~$n$, since solving a large-scale non-Hermitian Navier-Stokes problem requires additional preconditioning techniques that go beyond the purpose of this paper.)

The methods of Widlund and Rapoport are not directly applicable to \eqref{eqn:lin_nav}, since the Hermitian part $H=\begin{bmatrix}M -\frac{\tau}2 A_H & 0\\ 0 & 0\end{bmatrix}$ of $A$ is singular.  However, the system can be solved via Schur complement reduction (cf. Lemma~\ref{lem:schur}) by applying the four methods (Widlund, Rapoport, L-GMRES, and GMRES) to systems with the matrices
$$\widehat{A}_{11}=H_{11}+S_{11}=(M -\frac{\tau}2 A_H)+(-\frac{\tau}{2} A_S)\quad\mbox{and}\quad
{\cal S}_1=\frac{\tau^2}{4} B^*\widehat{A}_{11}^{-1}B,$$
which are both (non-Hermitian) positive definite. As in the Stokes equation in Section~\ref{num:Stokes}, we use $ \tau/2=10^{-4}$ and $\tau/2=10^{-3}$. We are here not interested in efficient ways to deal with the Schur complement ${\cal S}_1$, but in the performance of the iterative methods when applied to (non-Hermitian) positive definite matrices that depend on the step size parameter $\tau$. We therefore compute ${\cal S}_1$ exactly by inverting $\widehat{A}_{11}$ with MATLAB's backslash operator, which takes $4.087$s and $2.712$s for $ \tau/2=10^{-4}$ and $ \tau/2=10^{-3}$, respectively. In order to apply the methods of Widlund, Rapoport, and L-GMRES we compute incomplete Cholesky decompositions of $H_{11}$ and of the Hermitian part of ${\cal S}_1$ using MATLAB's {\tt ichol} function with drop tolerance $10^{-9}$. 

In Table~\ref{table:nav} we show the total time and number of iterative steps required by the different methods for solving the two systems with $\widehat{A}_{11}$ and ${\cal S}_1$, as well as the relative residual norm of the approximate solution of $Ax=b$ obtained in this way. The table also shows the corresponding values of (unpreconditioned) GMRES applied to the system $Ax=b$, denoted by GMRES($A$). Here we stopped the iteration after 500 steps, and we report the value of the relative residual norm attained at that point.

Figure~\ref{fig:nav-resid} shows the convergence curves of the four methods applied to the systems with $\widehat{A}_{11}$ and ${\cal S}_1$. The behavior is similar to what we have observed in Sections~\ref{num:Mbs} and~\ref{num:Stokes}, with the exception that in this example the method of Rapoport performs slightly better than the method of Widlund and L-GMRES. Again, these three methods outperform L-GMRES, and also GMRES($A$) is not competitive. In this example the (purely imaginary) eigenvalues of the matrix $K$ corresponding to $\widehat{A}_{11}$ are contained in the interval $i[-\lambda,\lambda]$ with
$$\mbox{$\lambda\approx 2.3516 \times 10^{-6}$ for $ \tau/2=10^{-4}$}\quad\mbox{and}\quad
\mbox{$\lambda\approx 1.0083 \times 10^{-4}$ for $\tau/2=10^{-3}$,}$$
and for the matrix $K$ corresponding to the Schur complement ${\cal S}_1$ they are contained in $i[-\lambda,\lambda]$ with
$$\mbox{$\lambda\approx 9.9225 \times 10^{-7}$ for $\tau/2=10^{-4}$}
\quad\mbox{and}\quad\mbox{$\lambda\approx 7.4483 \times 10^{-5} $ for $\tau/2=10^{-3}$.}$$
Using these very small values of $\lambda$ in the bounds \eqref{eqn:CGW-bound}, \eqref{eqn:Rap-bd}, and \eqref{eqn:L-GMRES-bd} explains the fast convergence of the methods of Widlund and Rapoport, and L-GMRES for the systems with $\widehat{A}_{11}$ and ${\cal S}_1$.

\begin{table}
\begin{tabular*}{\textwidth}{@{\extracolsep{\fill}}l rrr rrr}
	& \multicolumn{3}{c}{$\tau/2=10^{-4}$} & \multicolumn{3}{c}{$\tau/2=10^{-3}$} \\				
	\hline
	Method       &  	 Time & Iter. & $\norm{Rel. Res.}$ 	 &  	 Time  & Iter. & $\norm{Rel. Res.}$   	 \\
		\hline
		Widlund        &   0.388    &      8 & $	 8.802 \times 10^{-13} $  &  0.075    &  		   6 & $ 2.881 \times 10^{-13} $  \\
		Rapoport        &    0.366 	   &  	   3   & $ 8.758 \times 10^{-13} $   & 0.067    &  	     4 &$ 2.875 \times 10^{-13} $	  \\
		L-GMRES     &    	  0.464    &       9 & $	 2.114 \times 10^{-14} $	 &	 0.092       &  	  8 &  	$ 1.165 \times 10^{-13} $    \\
		GMRES     &  4.615 	    	    &   154 &   $9.191 \times 10^{-13} $    &   2.931      &    119 &   	$ 7.429 \times 10^{-13} $  \\
		\hline
		GMRES($A$)  &  3.657 	      &  	  500  &  $ 4.924 \times 10^{-05} $  &  3.402      &    	  500 & $ 1.561 \times 10^{-05} $
	\end{tabular*}
	\centering
	\caption{Linearized Navier-Stokes equation without stabilization. Running times, number of iterations, and relative residual norms at the final step for $\tau/2=10^{-4}$ and $\tau/2=10^{-3}$.}	
	\label{table:nav}
\end{table}

\begin{figure}
	\centering
	\includegraphics[width=0.49\textwidth]{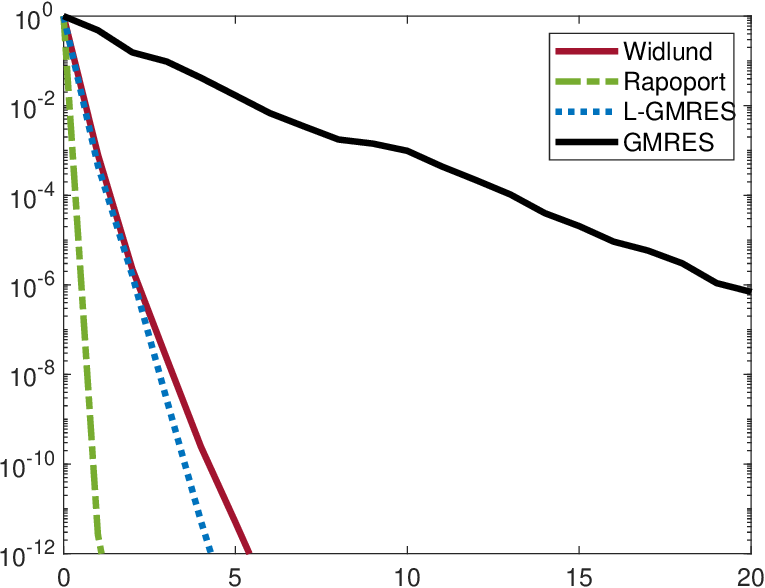}
	\includegraphics[width=0.49\textwidth]{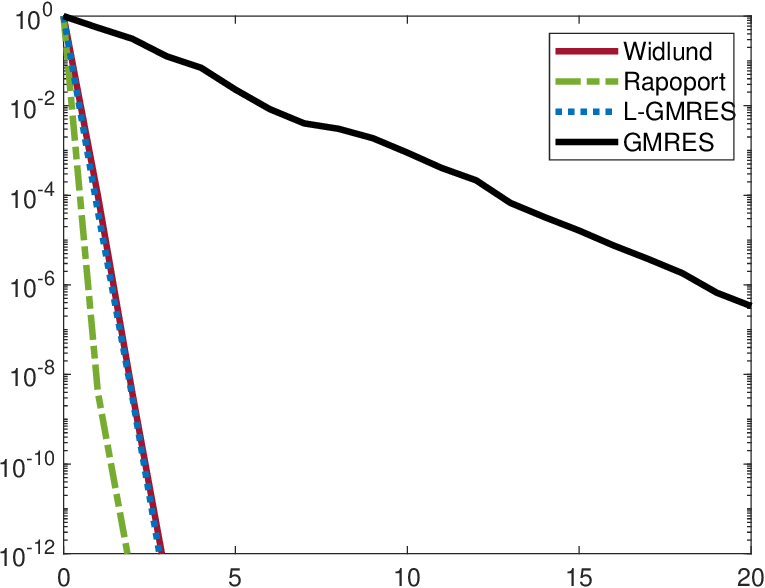}\\[2ex]
	\includegraphics[width=0.49\textwidth]{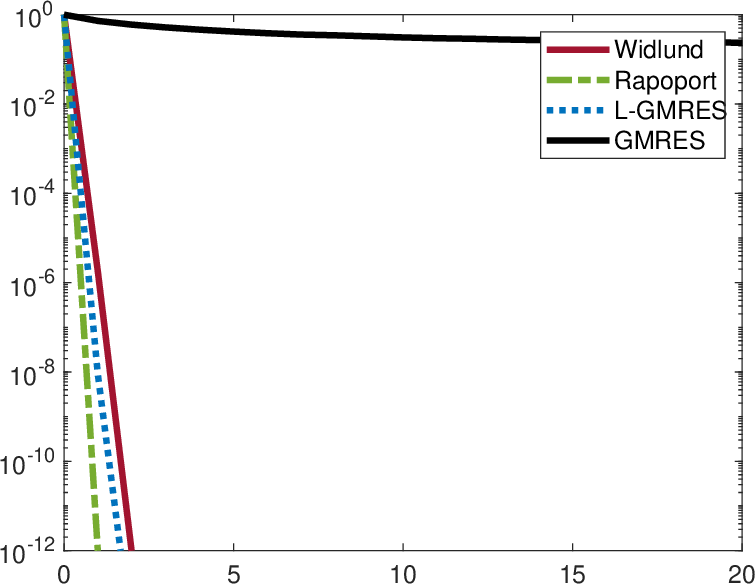}
	\includegraphics[width=0.49\textwidth]{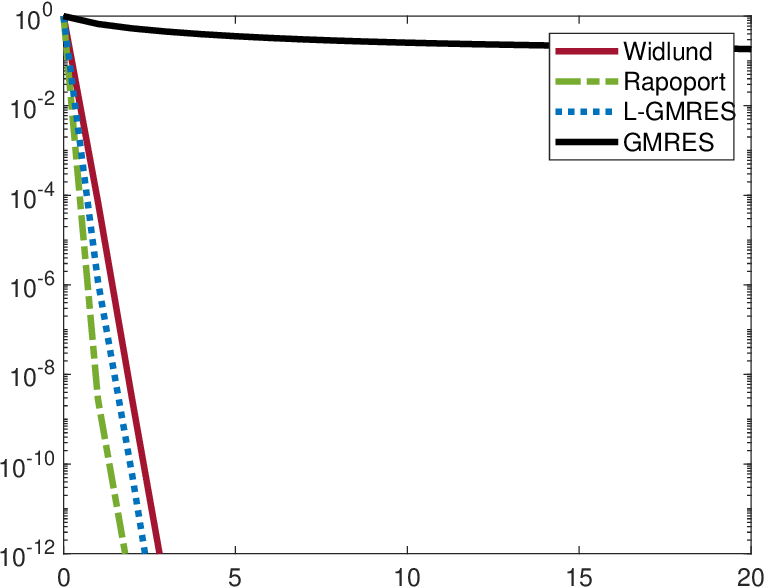}
	\caption{Linearized Navier-Stokes equation without stabilization. Relative residual norms of the four methods applied to the systems with $\widehat{A}_{11}$ (top row) and ${\cal S}_1$ (bottom row) with $ \tau/2=10^{-4}$ and $ \tau/2=10^{-3}$ (left and right).}
	\label{fig:nav-resid}
\end{figure}

\section{Concluding remarks}\label{sec:conclusion}
	
Dissipative Hamiltonian differential-algebraic equation (dHDAE) systems occur in a wide range of energy-based modeling applications, including thermodynamics, electromagnetics, and fluid mechanics. These systems can be classified using a staircase from, which reveals their differentiation index (either zero, one, or two). We have given a systematic overview of the three different cases. An important common feature is that the matrices arising in the (space and time) discretization of dHDAE systems split naturally into $A=H+S$, where the Hermitian part $H$ is positive definite or positive semidefinite. This feature can be exploited in the numerical solution of the corresponding linear algebraic systems.

A focus of our work has been the case of positive definite $H$, which allows the application of the Krylov subspace methods of Widlund and Rapoport. These methods were derived in the late 1970s, but have rarely been analyzed or even cited in the literature so far. We have summarized their main mathematical properties, and we have presented extensive numerical experiments with linear algebraic systems from different dHDAE application problems. In these experiments the three-term recurrence methods of Widlund and Rapoport have consistently outperformed L-GMRES and (unpreconditioned) GMRES. The behavior we have observed is consistent with the convergence bounds for the methods of Widlund and Rapoport, which indicate a fast convergence for the systems $(I+K)x=b$ when $K=H^{-1}S$ is ``small''. In time discretizations of dHDAE systems this important feature is virtually ``built in'', since the skew-Hermitian part of the dHDAE is being multiplied by the (usually) small time step parameter $\tau$.

Overall, we have therefore presented a holistic approach combining energy-based modeling using dHDAE systems, their structure-preserving discretization, and finally a structure-adapted linear algebraic computation.

The case of a positive semidefinite $H$ is challenging. We have shown in Lemma~\ref{lem:schur} that one can identify the ``singular part'' of $A=H+S$ via a unitary transformation, but this tool in not practical in large scale applications. However, in the mathematical modeling the block structure of the dHDAE frequently exposes the ``singular part'', and no further transformation is necessary. In such cases, we can apply the methods of Widlund and Rapoport to the ``positive definite part'' of the problem, and the ``singular part'' must be solved by other means. A closer analysis of the positive semidefinite case is a subject of future work.

\section*{Acknowledgments}
We thank the anonymous referees as well as Andreas Frommer and Karsten Kahl for helpful suggestions that have improved our presentation.

\bibliographystyle{siam}
\bibliography{glms21}

\end{document}